\documentclass[12pt,a4paper,reqno]{amsart}

\usepackage{amsthm,a4}
\usepackage[T1]{fontenc}
\usepackage[british]{babel}
\usepackage[unicode,psdextra]{hyperref}

\usepackage{mathtools}
\usepackage{xcolor}

\DeclareMathOperator{\Res}{Res}
\allowdisplaybreaks[4]

\newtheorem{lemma}{Lemma}
\newtheorem{theorem}[lemma]{Theorem}

\newtheorem{facts}[lemma]{Facts}
\newtheorem{example}[lemma]{Example}
\newtheorem{corollary}[lemma]{Corollary}

\begin{document}

	\title[Symmetry of meromorphic differentials, and integer
	partitions]{Symmetry of meromorphic differentials produced
		by involution identity, \\ and relation to integer partitions}

	\author{Alexander Hock}
	\author{Sergey Shadrin}
	\author{Raimar Wulkenhaar}
	
	\address{Institut für Mathematik,
		Ruprecht-Karls-Universität Heidelberg, \newline \hspace*{1.1em}%
		Im Neuenheimer Feld 205,
		69120 Heidelberg,
		Germany}
	\email{ahock@mathi.uni-heidelberg.de}
	
	\address{Korteweg--De Vriesinstituut voor Wiskunde, Universiteit van Amsterdam, \newline \hspace*{1.1em}%
		Postbus 94248, 1090GE Amsterdam, The Netherlands}
	\email{s.shadrin@uva.nl}
	
	\address{Mathematisches Institut, Universit\"at M\"unster,
		\newline \hspace*{1.1em}%
		Einsteinstr.\ 62,   48149 M\"unster, Germany}
	\email{raimar@math.uni-muenster.de}

	\begin{abstract}
		We prove that meromorphic differentials
		$\omega^{(0)}_n(z_1,...,z_n)$ which are recursively generated by an
		involution identity are symmetric in all their arguments
		$z_1,...,z_n$. The proof involves an intriguing combinatorial identity
		between integer partitions into given number of parts.
	\end{abstract}
	
	\subjclass[2010]{05A17, 30D05, 32A20}
	\keywords{Meromorphic forms on Riemann surfaces;
		Residue calculus; Involution; Integer partitions}

	\maketitle
	
	\section{Introduction}
	
	In \cite{Hock:2021tbl} the first and the third named authors have investigated a construction of
	meromorphic differentials $\omega^{(0)}_n(z_1,....,z_n)$ 
	on $n$ copies of $\mathbb{P}^1$ which is governed by a 
	holomorphic involution $\iota:\mathbb{P}^1\to\mathbb{P}^1$, $\iota^2z\equiv z$, 
	different from the identity.
	Such involutions are  necessarily of the form
	$\iota z =\frac{az+b}{cz-a}$ for complex numbers $a,b,c$ with $a^2+ bc\neq 0$.
	The construction starts with 
	\begin{align}
		\omega^{(0)}_{2}(w,z)&=B(w,z) -B(w,\iota z)\;,
		\label{om02}\\  
		&\text{where}\qquad 
		B(w,z) = \frac{dw\,dz}{(w-z)^2}
		\nonumber
	\end{align}
	is the Bergman kernel on $\mathbb{P}^1\times \mathbb{P}^1$, and
	proceeds recursively from an \emph{involution identity}
	\begin{align}
		&  \omega^{(0)}_{|I|+1}(z,I)
		+\omega^{(0)}_{|I|+1}(\iota z,I)
		\label{eq:flip-om}
		\\
		&=\sum_{s=2}^{|I|} \sum_{\substack{I_1\sqcup ...\sqcup I_s=I\\I_1,...,I_s\neq \emptyset}}
		\frac{1}{s} \Res\displaylimits_{q= z}  \Big(
		\frac{dy(z) dx(q)}{(y(z)-y(q))^{s}}  \prod_{j=1}^s
		\frac{\omega^{(0)}_{|I_j|+1}(q,I_j)}{dx(q)}
		\Big)\;,
		\nonumber
	\end{align}
	where $I=\{u_1,...,u_m\}$ and
	$|I|=m$.  Here
	\begin{itemize}
		\item $x:\mathbb{P}^1\to \mathbb{P}^1$ is a ramified covering with
		simple ramification points $\{\beta_i\}$; and such that $\iota$ does
		not fix or permute any ramification point of $x$, and that
		$\iota\beta_i$ is not a pole of $x$;

		\item $y(z)=-x(\iota z)$;

		\item the summation is over all partitions $I_1\sqcup ...\sqcup I_s=I$
		of $I=\{u_1,...,u_m\}$ into disjoint (and here non-empty) subsets
		$I_j$ of any order.
	\end{itemize}
	The first and the third named authors have proved in \cite{Hock:2021tbl} that, under mild assumptions on
	the location of poles, the $\omega^{(0)}_n$ satisfy linear and
	quadratic loop equations and obey blobbed topological recursion
	\cite{Borot:2015hna}. Loop equations in the quartic analogue of the
	Kontsevich model lead so such $\omega^{(0)}_n$ for the special
	involution $\iota z=-z$.  See \cite{Hock:2021tbl, Branahl:2020yru}.

	Although surprising at first sight, $\omega^{(0)}_2$ is symmetric: one
	has $B(w,\iota z)=B(\iota w,z)$ and $B(w,z)=B(\iota w,\iota z)$.  This
	paper addresses the important question whether all
	$\omega^{(0)}_n(z_1,....,z_n)$ defined by \eqref{om02} and
	\eqref{eq:flip-om} are symmetric in their arguments $z_i$. This
	symmetry is by no means obvious; it has not been proved so far. We
	will show that symmetry is reduced to the following purely combinatorial statement about
	integer partitions into a given number of parts:

\begin{theorem}
\label{conj:main}
  Let $\mathsf{P}_k(n)$ be the set of partitions of an integer
  $n\geq 1$ into $1\leq k\leq n$ parts.
  For any given integers $(s,k,l)$ with $k,l\geq 0$ and $s\geq \max(k+l,1)$
  and any given partition $\nu \in \mathsf{P}_s(2s-k-l)$ one has
\begin{align} \label{eq:Conjecture-Partitions}
  s!&=
  \sum_{\substack{(r-l)+p_1+...+p_l=s\\
   r-l\geq k,~p_1,...,p_l\geq 1}}~
  \sum_{\mu \in  \mathsf{P}_{r-l}(2r-k-2l)}\;
  \sum_{\rho_1 \in \mathsf{P}_{p_1}(2p_1-1)}\cdots
  \sum_{\rho_l \in \mathsf{P}_{p_l}(2p_l-1)}
  \nonumber
  \\
  &\times   \binom{\nu}{\mu, \rho_1,...,\rho_l} (r-l)!
(p_1-1)!\cdots (p_l-1)! \;.
\end{align}  
By
$\binom{\nu}{\mu_1,...,\mu_l}$ we denote the multinomial coefficient of partitions,
\begin{align*}
 & \binom{\nu}{\mu_1,...,\mu_l}=\lim_{N\to \infty} \prod_{k=1}^N
  \binom{h_k}{g_{k,1},...,g_{k,l}}
  \\
&\text{if} \qquad
  \nu=(...,\underbrace{2,...,2}_{h_2},  \underbrace{1,...,1}_{h_1})\;,\qquad
  \mu_i=(...,\underbrace{2,...,2}_{g_{2,i}},  \underbrace{1,...,1}_{g_{1,i}})\;,
\end{align*}
where $\binom{h_k}{g_{k,1},...,g_{k,l}}$ are the usual multinomial
coefficients which vanish unless $h_k= g_{k,1}+...+g_{k,l}$.
\end{theorem}
We give two examples at the end of sec.~\ref{sec:conjecture} and a proof of this theorem in sec.~\ref{sec:combinatorial-proof}. This theorem is a
main ingredient of the proof of the symmetry:

\begin{theorem} \label{thm:symmetry-differentials} All differentials
	$\omega^{(0)}_n(z_1,....,z_n)$ defined by \eqref{om02} and
	\eqref{eq:flip-om} are symmetric in their arguments $z_i$.
\end{theorem}

Theorem~\ref{thm:symmetry-differentials} is proved by reduction to Theorem~\ref{conj:main} in sec.~\ref{sec:conjecture}.



\section{Preliminaries}

In \cite{Hock:2021tbl} the first and the third named authors proved that the definition of $\omega^{(0)}_n$ via
(\ref{om02}) and (\ref{eq:flip-om}) implies (under mild assumptions) a representation
as sum over residues involving recursion kernels:
\begin{theorem}[\cite{Hock:2021tbl}]
\label{thm:flip}
Let $I=\{u_1,...,u_m\}$ of length $|I|:=m$. Assume that
$\omega^{(0)}_{|I|+1}(z,I)$ has poles at most in points where the rhs
of \eqref{eq:flip-om} has. Then there is a unique solution of
\eqref{om02} and \eqref{eq:flip-om} such that
$\omega^{(0)}_{|I|+1}(z,I)$ is for $|I| \geq 2$ holomorphic at any
$u_k$ and at any $\iota \beta_i$. This solution is recursively produced by
\begin{align}
\omega^{(0)}_{|I|+1}(z,I)
  &= \sum_{i}
\Res\displaylimits_{q= \beta_i}K_i(z,q)
\sum_{\substack{I_1\sqcup I_2=I\\ I_1,I_2\neq \emptyset}}
\omega^{(0)}_{|I_1|+1}(q,I_1)\omega^{(0)}_{|I_2|+1}(\sigma_i(q),I_2)
  \label{sol:omega}
  \\[-1ex]
  &-\sum_{k=1}^m d_{u_k}
 \Big[\Res\displaylimits_{q= \iota u_k}
\sum_{\substack{I_1\sqcup I_2=I\\ I_1,I_2\neq \emptyset}}
\tilde{K}(z,q,u_k)
d_{u_k}^{-1}\big( \omega^{(0)}_{|I_1|+1}(q,I_1)
\omega^{(0)}_{|I_2|+1}(q,I_2)\big)\Big]\,.
\nonumber
\end{align}
Here
$\sigma_i\neq \mathrm{id}$ denotes the local Galois involution
in the vicinity of $\beta_i$, i.e.\ $x(\sigma_i(z))=x(z)$,
$\lim_{z\to \beta_i}\sigma_i(z)=\beta_i$.
By
$d_{u_k}$ we denote the exterior differential in $u_k$, which on 1-forms
has a right inverse
given by the primitive $d^{-1}_u \omega(u)=\int_{u'=\infty}^{u'=u}\omega(u')
$.
The recursion kernels are given by
\begin{align}
K_i(z,q)&:=   \frac{\frac{1}{2} (\frac{dz}{z-q}-\frac{dz}{z-\sigma_i(q)})
}{dx(\sigma_i(q))(y(q)-y(\sigma_i(q)))}\;,\qquad
\nonumber
\\
\tilde{K}(z,q,u)&:=
\frac{\frac{1}{2}\big(\frac{d(\iota z)}{\iota z-\iota q}-\frac{d(\iota z)}{\
\iota z- u}\big)}{
  dx(q)(y(q)-y(\iota u))}\;.
\label{eq:kernel}
\end{align}
The solution \eqref{sol:omega}$+$\eqref{eq:kernel} implies symmetry of
\eqref{eq:flip-om} under $z\mapsto \iota z$.

\end{theorem}

We have slightly (but systematically) changed the notation. An
$\omega_{0,|I|+1}(I,q)$ in \cite{Hock:2021tbl} is in this paper called
$\omega^{(0)}_{|I|+1}(q,I)$. This change of order of the arguments is
merely notational and does not mean that $\omega^{(0)}_{|I|+1}$ is
symmetric. During the proof that (\ref{sol:omega}) follows from
(\ref{om02}) and (\ref{eq:flip-om}) --- under the given assumptions
--- a symmetry of the arguments of $\omega^{(0)}_{|I|+1}$ is never
used. The proof relies on well-known properties which are also
relevant for the present paper:
\begin{facts}
  \label{facts}
\begin{enumerate}
  \item Commutation of residues (not ramification points)
    \begin{align}
  \Res\displaylimits_{\tilde{q}=u} \Res\displaylimits_{q=u}
-\Res\displaylimits_{q=u}\Res\displaylimits_{\tilde{q}=u} 
=\Res\displaylimits_{q=u}\Res\displaylimits_{\tilde{q}=q} \;.
\label{commut-res-u}
\end{align}

\item Commutation of residues near ramification points:
\begin{align}
  \Res\displaylimits_{\tilde{q}= \beta_j}
  \Res\displaylimits_{q= \beta_j}
- \Res\displaylimits_{q= \beta_j}
  \Res\displaylimits_{ \tilde{q}= \beta_j}
  =  \Res\displaylimits_{q= \beta_j}
  \Big(\Res\displaylimits_{\tilde{q}= q}
 + \Res\displaylimits_{\tilde{q}= \sigma_j(q)} \Big)\;.
\label{commut-residues-beta-0}
\end{align}

\item Commutation of diffentials and residues; integration by parts:
  \begin{align}
    \Res\displaylimits_{q= u} d_u f(q,u,I) dq
    &=  d_u\Big[\Res\displaylimits_{q= u} f(q,u,I)dq\Big]\;,\qquad
\label{res-du}
    \\
    \Res\displaylimits_{q=  u} f(q,u,I') d_q g(q,u,I'')
    &=-   \Res\displaylimits_{q= u} g(q,u,I'') d_q f(q,u,I')\;.
    \nonumber
   \end{align}

 \item Invariance of residue under Galois conjugation:
   \begin{align}
     \label{inv-Galois}
  \Res\displaylimits_{q= \beta_j} f(q)dq
  = \Res\displaylimits_{q= \beta_j} f (\sigma_j(q)) d(\sigma(q))\;.
\end{align}

 \item Change of variables:
   \begin{align}
     \Res\displaylimits_{\iota z= \iota q} f(z,q)dz
     =\Res\displaylimits_{z= q} f(z,q)dz\;.
\label{change-iota}
   \end{align}  
   Indeed, the image of the circle $(\iota z)(t)=\iota q+re^{\mathrm{i}t}$
   is $z(t)=\iota (\iota q+re^{\mathrm{i}t})=q
   -\frac{(cq-a)^2 r e^{\mathrm{i}t}}{a^2+bc+cr (cq-a)e^{\mathrm{i}t}}$. For
   $r$ sufficiently small, this is a loop  about $q$.

 \item 
     For any holomorphic $f$, the difference $B(z,w)-B(f(z),f(w))$ is holomorphic at
     $z=w$.
     \label{Bergmman-f}
   
\end{enumerate}  
\end{facts}

The residue formula (\ref{sol:omega}) shows that
$\omega^{(0)}_{|I|+2}(z_1,z_2,u_1,...,u_m)$ is symmetric in its last
$m+1$ arguments.  To show complete symmetry of $\omega^{(0)}_{|I|+2}$
it is thus enough to show
$\omega^{(0)}_{|I|+2}(z_1,z_2,I)=\omega^{(0)}_{|I|+2}(z_2,z_1,I)$.  We
know from the explicit formulae in \cite{Hock:2021tbl} that
$\omega^{(0)}_3$ is completely symmetric.  We introduce projection
operators
\begin{align}
\mathcal{P}_{z;a}  \omega(z,\dots) &:=
  \Res\displaylimits_{q=a} \frac{dz}{z-q} \omega(q,\dots)
\end{align}
to the principal part of the Laurent series of a 1-form
$z\mapsto \omega(z,...)$ at $z=a$. This projection vanishes iff
$ \omega(z,\dots)$ is holomorphic at $z=a$.  By previous results about
the set of poles of $\omega^{(0)}_{|I|+2}$ (and a short verification
below) it is enough to prove
\begin{align}
0&=  \mathcal{P}_{z_1; \iota z_2} (\omega^{(0)}_{|I|+2}(z_1,z_2,I)
  -\omega^{(0)}_{|I|+2}(z_2,z_1,I))\;,
  \nonumber
  \\
0&=  \mathcal{P}_{z_1;\beta_j} (\omega^{(0)}_{|I|+2}(z_1,z_2,I)
-\omega^{(0)}_{|I|+2}(z_2,z_1,I))\;,
  \nonumber
\\
0&=  \mathcal{P}_{z_1;\iota u_1} (\omega^{(0)}_{|I|+2}(z_1,z_2,I)
-\omega^{(0)}_{|I|+2}(z_2,z_1,I))\;.
\label{Proj:12}
\end{align}
We prove these equations by induction, from the hypothesis that 
all $\omega^{(0)}_{m'+2}$ with $m'<m$ are symmetric. 
We easily
convince ourselves that $\omega^{(0)}_{|I|+2}(z_1,z_2,I)
-\omega^{(0)}_{|I|+2}(z_2,z_1,I)$ is holomorphic at any other point.
For instance, holomorphicity of
$\omega^{(0)}_{|I|+3}(z_1,z_2,u,I)$ at $z_1=u$ is built into
Thm~\ref{thm:flip}.  In $\omega^{(0)}_{|I|+3}(z_2,z_1,u,I)$, the
recursion formula (\ref{sol:omega}) for $z\mapsto z_2$ and
$I\mapsto \{z_1,u\}\cup I$ gives on the rhs terms
$\omega^{(0)}_{|I'|+3}(q,z_1,u,I')$ (or $u,z_1$ in different factors, but
they are holomorphic) for some $I'\subset I$ with $|I'|<|I |$. Since 
$\mathcal{P}_{z_1,u}$ commutes with all residues in (\ref{sol:omega}),
and
$\omega^{(0)}_{|I'|+3}(q,z_1,u,I')= \omega^{(0)}_{|I'|+3}(z_1,u,q,I')$
by induction hypothesis, $\omega^{(0)}_{|I|+3}(z_2,z_1,u,I)$ is also
holomorphic at $z_1=u$. A similar discussion gives
\begin{align}
  \mathcal{P}_{z_1,\iota \beta_j} \omega^{(0)}_{|I|+2}(z_2,z_1,I)\equiv
  \Res\displaylimits_{\tilde{q}=\iota \beta_j} \frac{dz_1}{z_1-\tilde{q}}
  \omega^{(0)}_{|I|+2}(z_2,\tilde{q},I)=0\;,
\label{no-res-iotabeta}
\end{align}
and we have $\mathcal{P}_{z_1,\iota \beta_j} \omega^{(0)}_{|I|+2}(z_1,z_2,I)$ by
construction. The same arguments show that 
$\omega^{(0)}_{|I|+2}(z_1,z_2,I)
-\omega^{(0)}_{|I|+2}(z_2,z_1,I)$ is holomorphic at $z_1=w$ for any 
$w \notin \{\iota z_2,\iota u_k,\beta_j\}$. If we can prove 
the three identities (\ref{Proj:12}), then $\omega^{(0)}_{|I|+2}(z_1,z_2,I)
-\omega^{(0)}_{|I|+2}(z_2,z_1,I)$ is holomorphic on $\mathbb{P}^1$, hence
identically zero. The last relation  (\ref{Proj:12}) requires that Theorem~\ref{conj:main} is true.

\section{The pole at $z_1=  \iota z_2$}

The following was already available in \cite{Hock:2021tbl}, but we find it
necessary to repeat the proof to show that a symmetry argument was never used.
\begin{lemma}
For $I\neq \emptyset$ one has
\begin{align}
&\mathcal{P}_{z;\iota u} \omega^{(0)}_{|I|+2}(z,u,I)
    \label{AHP-pole-uz}
\\
&=\mathcal{P}_{z;\iota u} \Big(
- d_u\Big[\sum_{s=1}^{|I|} \sum_{\substack{I_1\sqcup...\sqcup I_s=I\\
    I_1,...,I_s\neq \emptyset}}
  \frac{dy(\iota z)}{(y(\iota z)-y(u))^{s+1}} \prod_{i=1}^s
    \frac{\omega^{(0)}_{|I_i|+1}(u,I_i)}{dx(u)}\Big]\Big)\;.
    \nonumber
\end{align}
\begin{proof}
We apply the projection to (\ref{eq:flip-om}) in which we set 
$q\mapsto \iota \tilde{q}$, $z\mapsto \iota q$ and $I \mapsto I\cup \{u\}$.  
Since
$\omega^{(0)}_{|I|+2}(\iota q,u,I)$ is holomorphic at $q=\iota u$, we get   
  \begin{align*}
&\mathcal{P}_{z;\iota u} \omega^{(0)}_{|I|+2}(z,u,I)   
\equiv \Res\displaylimits_{q=\iota u} \frac{\omega^{(0)}_{|I|+2}(q,u,I)dz}{z-q}
\\[-1ex]
&=\Res\displaylimits_{q=\iota u} \frac{dz}{z-q}
\sum_{s=2}^{|I|+1} \sum_{\substack{I_1\sqcup ...\sqcup I_s=I \cup \{u\}\\
    I_1,...,I_s\neq \emptyset}}
\frac{1}{s} \Res\displaylimits_{\iota \tilde{q} = \iota q}  \Big(
\frac{dy(\iota q) dx(\iota \tilde{q})}{(y(\iota q)-y(\iota \tilde{q}))^{s}}
\prod_{j=1}^s
\frac{\omega^{(0)}_{|I_j|+1}(\iota \tilde{q},I_j)}{dx(\iota \tilde{q})}
\Big)\;.
\end{align*}
By (\ref{change-iota}) we can change the inner residue to
$\Res\displaylimits_{\tilde{q}= q}$, which allows to use
(\ref{commut-res-u}) for $u\mapsto \iota u$.  There is no contribution
if $\Res_{q=\iota u}$ is the inner residue.  If
$\Res_{\tilde{q}=\iota u}$ is the inner residue, the only contribution
is the case that $u$ arises in a term
$\omega^{(0)}_2(\iota \tilde{q},u)$, which can occur in $s$
places. Moving it in front cancels the factor $s$:
\begin{align*}
  &\mathcal{P}_{z;\iota u} \omega^{(0)}_{|I|+2}(z,u,I)
\\[-1ex]
&=
-\Res\displaylimits_{q=\iota u} \Res\displaylimits_{\tilde{q}\to \iota u}
\frac{dz}{z-q}
\sum_{s=1}^{|I|} \sum_{\substack{I_1\sqcup ...\sqcup I_s=I\\
    I_1,...,I_s\neq \emptyset}}
\Big(\frac{dy(\iota q) \omega^{(0)}_2(\iota \tilde{q},u)}{
  (y(\iota q)-y(\iota \tilde{q}))^{s+1}}
\prod_{j=1}^s
\frac{\omega^{(0)}_{|I_j|+1}(\iota \tilde{q},I_j)}{dx(\iota \tilde{q})}
\Big)\;.
\end{align*}
With (\ref{om02}) the only contribution is from the restriction
$\omega^{(0)}_2(\iota \tilde{q},u) \mapsto B(\iota \tilde{q},u)
\equiv B(\tilde{q},\iota u)
=d_u(\frac{d\tilde{q}}{\tilde{q}-\iota u})$.
The differential $d_u$ is moved outside the residues, see (\ref{res-du}).
The evaluation of the
inner residue is now straightforward. The outer residue is
just the projection $\mathcal{P}_{z;\iota u}$, and the assertion follows.
\end{proof}
\end{lemma}

Whereas $\mathcal{P}_{z_1;\iota z_2} \omega^{(0)}_{|I|+2}(z_1,z_2,I)$ is easily
deduced from the previous Lemma, the other situation
\begin{align}
\mathcal{P}_{z_1, \iota z_2}\omega^{(0)}_{|I|+2}(z_2,z_1,I)
=
\Res\displaylimits_{\tilde{q}= \iota z_2}
\frac{dz_1}{z_1-\tilde{q}} \omega^{(0)}_{|I|+2}(z_2,\tilde{q},I)
\label{om0I-12-b}
\end{align}
is more involved. 
We insert the recursion formula (\ref{sol:omega}) for $z\mapsto z_2$ and
$I\mapsto \{\tilde{q}\} \cup I$. Only the part that
captures the pole at $z_2=\iota \tilde{q}$ matters:
\begin{align}
&\mathcal{P}_{z_1,\iota z_2}\omega^{(0)}_{|I|+2}(z_2,z_1,I)
\label{P12-omega21}
\\
&=\Res\displaylimits_{\tilde{q}= \iota z_2}
  \frac{dz_1}{z_1-\tilde{q}}
  \Big(-2 d_{\tilde{q}}\Big[\Res\displaylimits_{q= \iota \tilde{q}}
  \tilde{K}(z_2,q,\tilde{q}) \!\!
\sum_{\substack{I'\sqcup I''=I\\ I''\neq \emptyset}} \!\!
  d_{\tilde{q}}^{-1}\big(\omega^{(0)}_{|I'|+2}(q,\tilde{q},I')\big)
   \omega^{(0)}_{|I''|+1}(q,I'')\Big]\Big)\;.
   \nonumber
\end{align}
Because of shorter length, $\omega^{(0)}_{|I'|+2}(q,\tilde{q},I')=
\omega^{(0)}_{|I'|+2}(\tilde{q},q,I')$ by induction hypothesis. The following
identity will often be used:
\begin{lemma}
\label{lemma:resK}
  For $|I|\geq 1$ one has
\begin{align}  
&2\Res\displaylimits_{q=\iota \tilde{q}} \tilde{K}(z_2,q,\tilde{q})
\sum_{\substack{I'\sqcup I''=I\\ I''\neq \emptyset}}
d_{\tilde{q}}^{-1}\Big(\omega^{(0)}_{|I'|+2}(\tilde{q},q,I')
\omega^{(0)}_{|I''|+1}(q,I'')\Big)
\label{ResKomom}
\\[-1ex]
&=
\Res\displaylimits_{q= \iota \tilde{q}} \frac{dz_2dq}{(z_2-q)^2}
\sum_{s=1}^{|I|}\frac{1}{s}
\sum_{\substack{I_1\sqcup...\sqcup I_s=I\\ I_1,...,I_s\neq \emptyset}}
\prod_{i=1}^s 
\frac{\omega^{(0)}_{|I_i|+1}(q,I_i)}{dx(q)(y(\iota \tilde{q})-y(q))}\;.
\nonumber
\end{align}
\begin{proof}
  We distinguish $I'=\emptyset$, which is directly treated with
  (\ref{om02}), from $I'\neq \emptyset$.  Since
  $\tilde{K}(z_2,q,\tilde{q})$ is holomorphic at $q=\iota \tilde{q}$,
  we can replace $\omega^{(0)}_{|I'|+2}(\tilde{q},q,I')$ for
  $I'\neq \emptyset$ by
  $\mathcal{P}_{\tilde{q};\iota
    q}\omega^{(0)}_{|I'|+2}(\tilde{q},q,I')$ given by
  (\ref{AHP-pole-uz}) (we rename $z\mapsto \tilde{q}$ and
  $u\mapsto q$; the integration variable will be $r$):
\begin{align*}  
  \text{(\ref{ResKomom})}
  &=2\Res\displaylimits_{q=\iota \tilde{q}} \tilde{K}(z_2,q,\tilde{q})
\Big(
  \frac{dq}{q-\tilde{q}}
  -
  \frac{d (\iota q)}{\iota q-\tilde{q}}\Big)
\omega^{(0)}_{|I|+1}(q,I)
\\
&-2\Res\displaylimits_{q=\iota \tilde{q}} \tilde{K}(z_2,q,\tilde{q})
\sum_{\substack{I'\sqcup I''=I\\ I',I''\neq \emptyset}}
\omega^{(0)}_{|I''|+1}(q,I'')
\\*[-1ex]
& \times d_{\tilde{q}}^{-1}d_q
\Big[
\Res\displaylimits_{r=\iota q} \frac{d\tilde{q}}{\tilde{q}-r}
\sum_{s=1}^{|I|} \sum_{\substack{I_1\sqcup...\sqcup I_s=I'\\ I_1,...,I_s\neq \emptyset}}
  \frac{dy(\iota r)}{(y(\iota r)-y(q))^{s+1}} \prod_{i=1}^s
    \frac{\omega^{(0)}_{|I_i|+1}(q,I_i)}{dx(q)}\Big]\;.
\end{align*}
In the first line only $\frac{d (\iota q)}{\iota q-\tilde{q}}$ contributes.
One has
\begin{align}
  \frac{dy(\iota z)}{(y(\iota z)-y(u))^{s+1}}
   =\frac{(y\circ \iota)'(z) dz}{((y \circ \iota)(z)-y(u))^{s+1}}
  =-\frac{1}{s}d_z\Big(\frac{1}{(y(\iota z)-y(u))^{s}}\Big)\;,
  \label{AHP-pole-diff}
\end{align}
which is inserted in the last line above for $z\mapsto r$ and
$u\mapsto q$. We integrate by parts in $r$ and cancel $d_{\tilde{q}}^{-1}$.
Then we rename $r\mapsto \iota r$ and change thanks to
(\ref{change-iota}) the resulting $\Res_{\iota r=\iota q}$ to
$\Res_{r=q}$. We also expand the recursion kernel (\ref{eq:kernel}):
\begin{align*}  
  \text{(\ref{ResKomom})}
  &= \Res\displaylimits_{q=\iota \tilde{q}}
\frac{d(\iota z_2)d(\iota q)}{(\iota z_2-\iota q)(\iota z_2-\tilde{q})}
  \frac{\omega^{(0)}_{|I|+1}(q,I)}{dx(q)(y(\iota \tilde{q})-y(q))}
  \\
  &+\Res\displaylimits_{q=\iota \tilde{q}}\Big\{ \Big(
  \frac{d(\iota z_2)}{\iota z_2-\iota q}
-\frac{d(\iota z_2)}{\iota z_2-\tilde{q}}\Big)
\sum_{\substack{I'\sqcup I''=I\\ I',I''\neq \emptyset}}
\frac{\omega^{(0)}_{|I''|+1}(q,I'')}{dx(q)(y(\iota \tilde{q})-y(q))}
\\*[-1ex]
& \times 
\Res\displaylimits_{r= q} \frac{d(\iota r)}{\iota r-\tilde{q}}
\sum_{s=1}^{|I'|} \frac{1}{s} \sum_{\substack{I_1\sqcup...\sqcup I_s=I'\\
    I_1,...,I_s\neq \emptyset}}
d_q\Big[\frac{1}{(y(r)-y(q))^{s}} \prod_{i=1}^s
    \frac{\omega^{(0)}_{|I_i|+1}(q,I_i)}{dx(q)}\Big]\Big\}\;.
\end{align*}
We express the double residue in the last two lines as commutator
(\ref{commut-res-u}) for $\tilde{q}\mapsto r$ and
$u\mapsto \iota \tilde{q}$. There is no contribution if
$\Res\displaylimits_{q=\iota \tilde{q}} $ is the inner residue,
whereas the inner residue $\Res\displaylimits_{r=\iota \tilde{q}} $ is
straightforward to integrate. In the first line, because of a
first-order pole, we can replace
$\frac{d(\iota z_2)d(\iota q)}{(\iota z_2-\iota q)(\iota
  z_2-\tilde{q})} \mapsto \frac{d(\iota z_2)d(\iota q)}{(\iota
  z_2-\iota q)^2} \equiv \frac{dz_2dq}{(z_2-q)^2}$ without changing
the residue:
\begin{align*}  
  \text{(\ref{ResKomom})}
  &=
\Res\displaylimits_{q=\iota \tilde{q}}
\frac{dz_2dq}{(z_2-q)^2}
  \frac{\omega^{(0)}_{|I|+1}(q,I)}{dx(q)(y(\iota \tilde{q})-y(q))}
  \\
  &-\Res\displaylimits_{q=\iota \tilde{q}}\Big\{ 
\Big(\frac{d(\iota z_2)}{\iota z_2-\iota q}-
\frac{d(\iota z_2)}{\iota z_2-\tilde{q}}\Big)
\sum_{\substack{I'\sqcup I''=I\\ I',I''\neq \emptyset}}
\frac{\omega^{(0)}_{|I''|+1}(q,I'')}{dx(q)(y(\iota \tilde{q})-y(q))}
\tag{*}
\\*[-1ex]
& \times 
\sum_{s=1}^{|I'|} \frac{1}{s} \sum_{\substack{I_1\sqcup...\sqcup I_s=I'\\
    I_1,...,I_s\neq \emptyset}}
d_q\Big[\frac{1}{(y(\iota \tilde{q})-y(q))^{s}} \prod_{i=1}^s
    \frac{\omega^{(0)}_{|I_i|+1}(q,I_i)}{dx(q)}\Big]\Big\}\;.
    \tag{**}
  \end{align*}
  We convince ourselves that the factor
  $\frac{\omega^{(0)}_{|I''|+1}(q,I'')}{dx(q)(y(\iota
    \tilde{q})-y(q))}$ in the line (*) can be moved inside the
  differential $d_q$ of the last line (**) at expense of adjusting
  $\frac{1}{s}\mapsto \frac{1}{s+1}$. Consider an ordered partition
  $I=J_1\sqcup ...\sqcup J_{s+1}$ where $J_1<J_2<...<J_{s+1}$ in some
  lexicographic order. Let
  $A_J(q)=\frac{\omega^{(0)}_{|J|+1}(q,J)}{dx(q)(y(\iota
    \tilde{q})-y(q))}$.  Then the $A_{J_k}$ with $1\leq k\leq s+1$
  arise in the lines (*) and (**) as
\begin{align}
  \frac{s!}{s}  \sum_{k=1}^{s+1} A_{J_k}(q) d_q
  \Big(A_{J_1}(q) \stackrel{k}{\check{\cdots}}A_{J_{s+1}}(q)\Big)\;,
\label{A-order}
\end{align}
where the prefactor $\frac{1}{s}$ is the same as in the last line (**)
and the factor $s!$ takes into account the different distributions of
the ordered $J_l$ with $l\neq k$ to the unordered partition
$I_1\sqcup ...\sqcup I_s$. Every $dA_l$ arises $s$ times in
(\ref{A-order}), so
\begin{align*}
  \eqref{A-order}
=  s!  d_q\big(A_{J_1}(q) ...A_{J_{s+1}}(q)\big)\;.
\end{align*}
We write $s!=\frac{(s+1)!}{s+1}$ and reabsorb $(s+1)!$ into the sum
over unordered partitions. The factor $\frac{1}{s+1}$ remains as
claimed; $s+1$ is the total number of factors $A_{J_k}(q)$, and this
number is $s+1\geq 2$.  Shifting $s+1\mapsto s$, now with $s\geq 2$,
gives after integration by parts
\begin{align*}  
  \text{(\ref{ResKomom})}
  &=
\Res\displaylimits_{q=\iota \tilde{q}}
\frac{dz_2dq}{(z_2-q)^2}
  \frac{\omega^{(0)}_{|I|+1}(q,I)}{dx(q)(y(\iota \tilde{q})-y(q))}
  \\
  &+\Res\displaylimits_{q=\iota \tilde{q}} 
\frac{d(\iota z_2)d(\iota q)}{(\iota z_2-\iota q)^2}
\sum_{s=2}^{|I|} \frac{1}{s} \sum_{\substack{I_1\sqcup...\sqcup I_s=I'\\
    I_1,...,I_s\neq \emptyset}}
\prod_{i=1}^s
    \frac{\omega^{(0)}_{|I_i|+1}(q,I_i)}{dx(q)(y(\iota \tilde{q})-y(q))}\;.
  \end{align*}
  We have $B(\iota z_2,\iota q)=B(z_2,q)$, and then the first line
  completes this sum to $s=1$. This is the assertion.
\end{proof}
\end{lemma}

Equation (\ref{P12-omega21}) takes with Lemma \ref{lemma:resK} the form
\begin{align*}
&\mathcal{P}_{z_1,\iota z_2}\omega^{(0)}_{|I|+2}(z_2,z_1,I)
\\
&=-\Res\displaylimits_{\tilde{q}= \iota z_2}
  \frac{dz_1}{z_1-\tilde{q}}
  d_{\tilde{q}}\Big[
\Res\displaylimits_{q=\iota \tilde{q}} 
\frac{d z_2 dq}{(z_2-q)^2}
\sum_{s=1}^{|I|} \frac{1}{s} \sum_{\substack{I_1\sqcup...\sqcup I_s=I\\
    I_1,...,I_s\neq \emptyset}}
\prod_{i=1}^s
    \frac{\omega^{(0)}_{|I_i|+1}(q,I_i)}{dx(q)(y(\iota \tilde{q})-y(q))}
    \Big]
\\
&=d_{z_2} \Big[\Res\displaylimits_{\tilde{q}= \iota z_2}
  \Res\displaylimits_{q= \iota \tilde{q}}
 \frac{dz_1}{(z_1-\tilde{q})}
 \frac{d q}{(q-z_2)}
\sum_{s=1}^{|I|} \frac{dy(\iota \tilde{q})}{(y(\iota \tilde{q})-y(q))^{s+1}}
\sum_{\substack{I_1\sqcup...\sqcup I_s=I\\ I_1,...,I_s\neq \emptyset}} \!
\prod_{i=1}^s
    \frac{\omega^{(0)}_{|I_i|+1}(q,I_i)}{dx(q)}
    \Big]\,.  
\nonumber
\end{align*}
We rename $\tilde{q}\mapsto \iota \tilde{q}$, take (\ref{change-iota})
into account and write the residues as a commutator
(\ref{commut-res-u}) with $u\mapsto z_2$ and
$q\leftrightarrow \tilde{q}$. There is no contribution if
$\Res\displaylimits_{\tilde{q}= z_2}$ is the inner residue, whereas
the inner residue $\Res\displaylimits_{q= z_2}$ is immediately
evaluated:
\begin{align*}
&\mathcal{P}_{z_1, \iota z_2}\omega^{(0)}_{|I|+2}(z_2,z_1,I)
\\
&=-d_{z_2} \Big[
\Res\displaylimits_{\tilde{q}= z_2}
\frac{dz_1}{z_1-\tilde{q}}
\sum_{s=1}^{|I|} \frac{dy(\tilde{q})}{(y(\tilde{q})-y(z_2))^{s+1}}
\sum_{\substack{I_1\sqcup...\sqcup I_s=I\\ I_1,...,I_s\neq \emptyset}}
\prod_{i=1}^s
    \frac{\omega^{(0)}_{|I_i|+1}(z_2,I_i)}{dx(z_2)}
    \Big]\;.  
\nonumber
\end{align*}
Renaming $\tilde{q}\mapsto \iota \tilde{q}$ and changing the resulting
$\Res_{\iota \tilde{q}= \iota^2 z_2}$ to
$\Res_{\tilde{q}= \iota z_2}$ by (\ref{change-iota}) we get 
exactly $\mathcal{P}_{z_1,\iota z_2}\omega^{(0)}_{|I|+2}(z_1,z_2,I)$
when using  (\ref{AHP-pole-uz}) for $z\mapsto z_1$ and $u\mapsto z_2$.
In other words, $\omega^{(0)}_{|I|+2}(z_1,z_2,I)-\omega^{(0)}_{|I|+2}(z_2,z_1,I)$
is holomorphic at $z_1=\iota z_2$.

\section{The pole at $z_1=\beta_j$}

The projection to the principal part of the Laurent series of
$\omega^{(0)}_{|I|+2}(z_1,z_2,I)$ at $z_1=\beta_j$, where $\beta_j$ is
a ramification point $dx(\beta_j)=0$, is given by the case $i=j$ of
the first line of the residue formula (\ref{sol:omega}). Taking
invariance (\ref{inv-Galois}) under Galois conjugation
and $K_j(z,q)=K_j(z,\sigma_j(q))$ into account,
this projection reads
\begin{align}
\mathcal{P}_{z_1;\beta_j} \omega^{(0)}_{|I|+2}(z_1,z_2,I)
  &=2\Res\displaylimits_{q=\beta_j} \;
  K_j(z_1,q)
  \Big( \omega^{(0)}_2(q,z_2)  \omega^{(0)}_{|I|+1}(\sigma_j(q),I)
  \nonumber
  \\[-1ex]
  &\qquad\qquad + \sum_{\substack{I_1\sqcup I_2=I\\ I_1,I_2\neq \emptyset}}
  \omega^{(0)}_{|I_1|+2}(q,z_2,I_1)  \omega^{(0)}_{|I_2|+1}(\sigma_j(q),I_2)
\Big)\;.
\label{om0I-beta-target}
\end{align}
More involved is 
$\mathcal{P}_{z_1;\beta_j} \omega^{(0)}_{|I|+2}(z_2,z_1,I)
=\Res_{\tilde{q}=\beta_j}\frac{\omega^{(0)}_{|I|+2}(z_2,\tilde{q},I)dz_1}{
  z_1-\tilde{q}}$. We represent
$\omega^{(0)}_{|I|+2}(z_2,\tilde{q},I)$ by the recursion formula
(\ref{sol:omega}) for $z\mapsto z_2$ and
$I\mapsto \{\tilde{q}\} \cup I$ and divide its projection into
$ \mathcal{P}_{z_1:\beta_j} \omega^{(0)}_{|I|+2}(z_2,z_1,I)
=(\mathcal{P}_{z_1;\beta_j} \omega^{(0)}_{|I|+2}(z_2,z_1,I))_{(a)}+
(\mathcal{P}_{z_1;\beta_j} \omega^{(0)}_{|I|+2}(z_2,z_1,I))_{(b)} $
with
\begin{align}
  &(\mathcal{P}_{z_1;\beta_j} \omega^{(0)}_{|I|+2}(z_2,z_1,I))_{(a)}
\label{om0I-beta-a}
\\
&=2\Res\displaylimits_{\tilde{q}=\beta_j}
\frac{dz_1}{z_1-\tilde{q}}
\sum_{i\neq j} \Res\displaylimits_{q=\beta_i} \;
  K_i(z_2,q)
\sum_{\substack{I_1\sqcup I_2=I\\ I_2\neq \emptyset}}
 \omega^{(0)}_{|I_1|+2}(q,\tilde{q},I_1)  \omega^{(0)}_{|I_2|+1}(\sigma_i(q),I_2)
\nonumber
\\
&-2\sum_{k=1}^{|I|}\Res\displaylimits_{\tilde{q}=\beta_j}
\frac{dz_1}{z_1-\tilde{q}}
d_{u_k} \Big[\Res\displaylimits_{q= \iota u_k} \tilde{K}(z_2,q,u_k)
d_{u_k}^{-1}
\Big(
\sum_{\substack{I_1\sqcup I_2=I\\ I_2\neq \emptyset}} \!\!\!
 \omega^{(0)}_{|I_1|+2}(q,\tilde{q},I_1) \omega^{(0)}_{|I_2|+1}(q,I_2)
  \Big)\Big]
\nonumber
\end{align}
and
\begin{align}
&(\mathcal{P}_{z_1;\beta_j} \omega^{(0)}_{|I|+2}(z_2,z_1,I))_{(b)}
\label{om0I-1beta-b}
\\
&=\Res\displaylimits_{\tilde{q}=\beta_j}
\frac{dz_1}{z_1-\tilde{q}}\Big\{
2\Res\displaylimits_{q=\beta_j} \;
  K_j(z_2,q)
\sum_{\substack{I_1\sqcup I_2=I\\ I_2\neq \emptyset}}
  \omega^{(0)}_{|I_1|+2}(q,\tilde{q},I_1)
  \omega^{(0)}_{|I_2|+1}(\sigma_j(q),I_2)
\nonumber
\\
&+\Res\displaylimits_{\tilde{q}=\beta_j}
\frac{dz_1}{z_1-\tilde{q}}
\Res\displaylimits_{q=\iota \tilde{q}} \frac{dz_2}{z_2-q}
\omega^{(0)}_{|I|+2}(q,\tilde{q},I)\;.
\nonumber
\end{align}
In the very last line we have not expanded the projection to the pole
at $q=\iota \tilde{q}$; this form will be more convenient.  In
$ (\mathcal{P}_{z_1;\beta_j} \omega^{(0)}_{|I|+2}(z_2,z_1,I))_{(a)}$
the two residues do not interfere with each other; they commute. In
the now inner integral over $\tilde{q}$,
$\omega^{(0)}_2(q,\tilde{q},I_1)$ is holomorphic at
$\tilde{q}=\beta_j$ for $I_1=\emptyset$ so that the sums restrict to
$I_1,I_2 \neq \emptyset$. Because of shorter length, the induction
hypothesis gives
$\omega^{(0)}_{|I_1|+2}(q,\tilde{q},I_1)=
\omega^{(0)}_{|I_1|+2}(\tilde{q},q,I_1)$.  Now the residue
$\Res_{\tilde{q}=\beta_j}$ can be expressed by the corresponding part
of the recursion formula (\ref{sol:omega}) for
$\mathcal{P}_{z_1;\beta_j}\omega^{(0)}_{|I_1|+2}(z_1,q,I_1)$.  The
residues commute again, and after relabelling the index sets we get
\begin{align*}
  &(\mathcal{P}_{z_1;\beta_j} \omega^{(0)}_{|I|+2}(z_2,z_1,I))_{(a)}
\\
&=2 \Res\displaylimits_{\tilde{q}=\beta_j}K_j(z_1,\tilde{q})
\bigg\{
2\sum_{i\neq j}  
 \Res\displaylimits_{q=\beta_i} \;
  K_i(z_2,q) 
  \nonumber
  \\[-1ex]
  &\qquad\qquad \times 
\sum_{\substack{I_1\sqcup I_2 \sqcup I_3 =I\\ I_2,I_3 \neq \emptyset}}
\omega^{(0)}_{|I_1|+2}(\tilde{q},q,I_1)
\omega^{(0)}_{|I_2|+1}(\sigma_i(q),I_2)
\omega^{(0)}_{|I_3|+1}(\sigma_j(\tilde{q}),I_3)  
\nonumber
\\
&-2\sum_{k=1}^{|I|}
d_{u_k} \Big[
\Res\displaylimits_{q=\iota u_k} \tilde{K}(z_2,q,u_k)
\nonumber
  \\[-0.5ex]
  &\qquad\qquad \times 
\sum_{\substack{I_1\sqcup I_2\sqcup I_3=I\\ I_2,I_3\neq \emptyset}}
d_{u_k}^{-1} \Big(\omega^{(0)}_{|I_1|+3}(\tilde{q},q,I_1)
\omega^{(0)}_{|I_2|+1}(q,I_2)
\omega^{(0)}_{|I_3|+1}(\sigma_j(\tilde{q}),I_3)
\Big)\Big]
\bigg\}\;.
\end{align*}
Next, we rearrange the decompositions $I_1\sqcup I_2\sqcup I_3=I$ into
a first summation over $(I\setminus I_3) \sqcup I_3$ with
$I_3\neq \emptyset$ and then another split of $(I\setminus I_3)$. In
the last line we have $u_k \in I\setminus I_3$; otherwise there is no
pole at $q=\iota u_k$.  Since $I_3\neq \emptyset$, all
$\omega^{(0)}_{|I_1|+2}(\tilde{q},q,I_1)$ in the above equation are
symmetric by induction hypothesis, and we can change them to
$\omega^{(0)}_{|I_1|+2}(q,\tilde{q},I_1)$.  We then notice that the
terms in $\{~\}$ above, with all
$\omega^{(0)}_{|I_3|+1}(\sigma_j(\tilde{q}),I_3)$ taken out, almost
assemble to the recursion formula (\ref{sol:omega}) for
$\omega^{(0)}_{|I\setminus I_3|+2}(z_2,\tilde{q},I\setminus I_3)$;
with two terms missing: the residue at $q=\beta_j$ and the residue at
$q=\iota \tilde{q}$:
\begin{align}
  &(\mathcal{P}_{z_1;\beta_j} \omega^{(0)}_{|I|+2}(z_2,z_1,I))_{(a)}
\label{om0I-beta-a1}
\\
&=2 \Res\displaylimits_{\tilde{q}=\beta_j}K_j(z_1,\tilde{q})
\sum_{\substack{I_3\subset I \\ \emptyset \neq I_3 \neq I}}
\omega^{(0)}_{|I\setminus I_3|+2}(z_2,\tilde{q},I\setminus I_3) 
\omega^{(0)}_{|I_3|+1}(\sigma_j(\tilde{q}),I_3)
\tag{*}
\\
&
-4 \Res\displaylimits_{\tilde{q}=\beta_j}
 \Res\displaylimits_{q=\beta_j} 
K_j(z_1,\tilde{q})  K_j(z_2,q) 
  \nonumber
  \\[-1ex]
  &\qquad\qquad \times 
\sum_{\substack{I_1\sqcup I_2 \sqcup I_3 =I\\ I_2,I_3 \neq \emptyset}}
\omega^{(0)}_{|I_1|+2}(\tilde{q},q,I_1)
\omega^{(0)}_{|I_2|+1}(\sigma_j(q),I_2)
\omega^{(0)}_{|I_3|+1}(\sigma_j(\tilde{q}),I_3)  
\nonumber
\\
&- 2 \Res\displaylimits_{\tilde{q}=\beta_j}K_j(z_1,\tilde{q})
\sum_{\substack{I_3\subset I \\ \emptyset \neq I_3 \neq I}}
\omega^{(0)}_{|I_3|+1}(\sigma_j(\tilde{q}),I_3)
\Res\displaylimits_{q= \iota \tilde{q}} \frac{dz_2}{z_2-q}
\omega^{(0)}_{|I\setminus I_3|+2}(q,\tilde{q},I\setminus I_3) \;.
\nonumber
\end{align}
In the second line (*) of (\ref{om0I-beta-a1}) we have
$\omega^{(0)}_{|I\setminus I_3|+2}(z_2,\tilde{q},I\setminus I_3)
=\omega^{(0)}_{|I\setminus I_3|+2}(\tilde{q},z_2,I\setminus I_3)$ by
induction hypothesis. Now the line (*) almost coincides with the
projection $\mathcal{P}_{z_1;\beta_j} \omega^{(0)}_{|I|+2}(z_1,z_2,I)$
given in (\ref{om0I-beta-target}); only the first term on the rhs,
corresponding to $I_3=I$ in (*), is missing:
\begin{align}
  \text{(\ref{om0I-beta-a1}*)}
  =\mathcal{P}_{z_1;\beta_j} \omega^{(0)}_{|I|+2}(z_1,z_2,I)
  - 2 \Res\displaylimits_{\tilde{q}=\beta_j}K_j(z_1,\tilde{q})
\omega^{(0)}_{2}(z_2,\tilde{q})
\omega^{(0)}_{|I|+1}(\sigma_j(\tilde{q}),I)\;.
\label{om0I-beta-a2}
\end{align}

We continue with (\ref{om0I-1beta-b}). In the middle line we commute the 
residues with (\ref{commut-residues-beta-0}). Here 
$\omega^{(0)}_2(q,\tilde{q})$, i.e.\ the case $I_1=\emptyset$,
only contributes to the inner residue
$\Res\displaylimits_{\tilde{q}= q}$. For $I_1\neq \emptyset$ we have
  $\omega^{(0)}_{|I_1|+2}(q,\tilde{q},I_1)=\omega^{(0)}_{|I_1|+2}(\tilde{q},q,I_1)
  $ by induction hypothesis, so that this term
  only contributes to the inner residue
  $\Res\displaylimits_{\tilde{q}= \beta_j}$:
\begin{align}
&(\mathcal{P}_{z_1;\beta_j} \omega^{(0)}_{|I|+2}(z_2,z_1,I))_{(b)}
-\Res\displaylimits_{\tilde{q}=\beta_j}
\frac{dz_1}{z_1-\tilde{q}}
\Res\displaylimits_{q=\iota \tilde{q}} \frac{dz_2}{z_2-q}
\omega^{(0)}_{|I|+2}(q,\tilde{q},I)
 \label{om0I-1beta-b2}
\\
&=2
\Res\displaylimits_{q=\beta_j} 
K_j(z_2,q)
\omega^{(0)}_{|I|+1}(\sigma_j(q),I)
\Res\displaylimits_{\tilde{q}=q}
\frac{dz_1}{z_1-\tilde{q}}
\omega^{(0)}_2(q,\tilde{q})  
\nonumber
\\[-0.5ex]
  &+
 2 \Res\displaylimits_{q=\beta_j} 
  K_j(z_2,q)
\sum_{\substack{I_1\sqcup I_2=I\\ I_1,I_2\neq \emptyset}}
  \omega^{(0)}_{|I_2|+1}(\sigma_j(q),I_2)
\Res\displaylimits_{\tilde{q}=\beta_j}
\frac{dz_1}{z_1-\tilde{q}}
  \omega^{(0)}_{|I_1|+2}(\tilde{q},q,I_1)
\nonumber
\\
&=2
\Res\displaylimits_{q=\beta_j} 
K_j(z_2,q) 
\frac{dz_1dq}{(z_1-q)^2}
\omega^{(0)}_{|I|+1}(\sigma_j(q),I)
\nonumber
  \\
  &+
 4 \Res\displaylimits_{q=\beta_j} \Res\displaylimits_{\tilde{q}=\beta_j}
  K_j(z_2,q)K_j(z_1,\tilde{q})
  \nonumber
  \\
  &\qquad\qquad \times
  \sum_{\substack{I_1\sqcup I_2\sqcup I_3=I\\ I_2,I_3\neq \emptyset}}
  \omega^{(0)}_{|I_1|+2}(\tilde{q},q,I_1)
    \omega^{(0)}_{|I_2|+1}(\sigma_j(\tilde{q}),I_2)
  \omega^{(0)}_{|I_3|+1}(\sigma_j(q),I_3)\;.
\nonumber
\end{align}
The last two lines were obtained by inserting the case $i=j$ of the
first line of the recursion formula (\ref{sol:omega}) which provides the
projection to the principal part of the Laurent series at
$z_1=\beta_j$. We also relabelled the partition of $I$. We combine
(\ref{om0I-beta-a1}), (\ref{om0I-beta-a2}) and 
(\ref{om0I-1beta-b2}) to
\begin{align}
  &\mathcal{P}_{z_1;\beta_j} \omega^{(0)}_{I|+2}(z_2,z_1,I)
  -\mathcal{P}_{z_1;\beta_j} \omega^{(0)}_{|I|+2}(z_1,z_2,I)
  \label{om0I-1beta-diff}
  \\
&  = \Delta_j(z_1,z_2,I)
  - 2 \Res\displaylimits_{\tilde{q}=\beta_j}K_j(z_1,\tilde{q})
\omega^{(0)}_{2}(z_2,\tilde{q})
\omega^{(0)}_{|I|+1}(\sigma_j(\tilde{q}),I)
\tag{**}
\\
&- 2 \Res\displaylimits_{\tilde{q}=\beta_j}K_j(z_1,\tilde{q})
\sum_{\substack{I_3\subset I \\ \emptyset \neq I_3 \neq I}}
\omega^{(0)}_{|I_3|+1}(\sigma_j(\tilde{q}),I_3)
\Res\displaylimits_{q=\iota \tilde{q}} \frac{dz_2}{z_2-q}
\omega^{(0)}_{|I\setminus I_3|+2}(q,\tilde{q},I\setminus I_3) 
\nonumber
\\
&+2
\Res\displaylimits_{q=\beta_j} 
K_j(z_2,q) 
\frac{dz_1dq}{(z_1-q)^2}
\omega^{(0)}_{|I|+1}(\sigma_j(q),I)
\tag{***}
  \\
&+\Res\displaylimits_{\tilde{q}=\beta_j}
\frac{dz_1}{z_1-\tilde{q}}
\Res\displaylimits_{q=\iota \tilde{q}} \frac{dz_2}{z_2-q}
\omega^{(0)}_{|I|+2}(q,\tilde{q},I)\;,
\nonumber
\end{align}
where
\begin{align*}
\Delta_j(z_1,z_2,I) &:=
4 \Big(\Res\displaylimits_{q=\beta_j}
\Res\displaylimits_{\tilde{q}=\beta_j}
-
\Res\displaylimits_{\tilde{q}=\beta_j}
 \Res\displaylimits_{q=\beta_j} 
\Big)
K_j(z_1,\tilde{q})  K_j(z_2,q) 
  \nonumber
  \\[-1ex]
  &\qquad\quad \times 
\sum_{\substack{I_1\sqcup I_2 \sqcup I_3 =I\\ I_2,I_3 \neq \emptyset}} \!\!\!
\omega^{(0)}_{|I_1|+2}(\tilde{q},q,I_1)
\omega^{(0)}_{|I_2|+1}(\sigma_j(q),I_2)
\omega^{(0)}_{|I_3|+1}(\sigma_j(\tilde{q}),I_3)  \;.
\end{align*}
The further treatment of $\Delta_j(z_1,z_2,I)$ parallels the proof of
\cite[Thm 4.6]{Eynard:2007kz}.  The commutation rule (\ref{commut-residues-beta-0}), with
$q\leftrightarrow \tilde{q}$, leads to an inner residue at
$\Res_{q=\tilde{q}}$ in $\Delta_j(z_1,z_2,I)$, which is only present if
$I_1=\emptyset$, and then only the Bergman kernel  $B(\tilde{q},q)$ of
$\omega^{(0)}_2(\tilde{q},q)$ in (\ref{om02})
contributes to the pole:
\begin{align*}
  &\Delta_j(z_1,z_2,I)
  \\
  &=
4
\Res\displaylimits_{\tilde{q}=\beta_j}
 \Res\displaylimits_{q=\tilde{q}} 
K_j(z_1,\tilde{q})  K_j(z_2,q) B(\tilde{q},q)
\sum_{\substack{I_2 \sqcup I_3 =I\\ I_2,I_3 \neq \emptyset}}
\omega^{(0)}_{|I_2|+1}(\sigma_j(q),I_2)
\omega^{(0)}_{|I_3|+1}(\sigma_j(\tilde{q}),I_3)  \;.
\end{align*}
By Facts \ref{facts}.(\ref{Bergmman-f}) we 
can replace $B(\tilde{q},q)\mapsto
\frac{1}{2}B(\tilde{q},q)+\frac{1}{2}B(\sigma_j(\tilde{q}),\sigma_j(q))$
without changing the inner residue. 
In the part of the integrand containing 
$\frac{1}{2}B(\sigma_j(\tilde{q}),\sigma_j(q))$ we can change,
by invariance (\ref{inv-Galois}) of the residue under Galois conjugation, 
$\tilde{q}\leftrightarrow \sigma_j(\tilde{q})$. We also rename 
$q\leftrightarrow \sigma_j(q)$. This brings the inner residue to 
$\Res_{\sigma_j(q)=\sigma_j(\tilde{q})}$. But $\sigma_j$ is locally an involution so that 
the property (\ref{change-iota}) also holds for $\iota\mapsto \sigma_j$.
These steps result with $K_j(z,\sigma_j(q))\equiv K_j(z,q)$
in
\begin{align*}
\Delta_j(z_1,z_2,I)
&=2\Res\displaylimits_{\tilde{q}=\beta_j} K_j(z_1,\tilde{q})
\Res\displaylimits_{q=\tilde{q}}
K_j(z_2,q)B(q, \tilde{q})
\Big(U(q,\tilde{q};I)+U(\sigma_j(q),\sigma_j(\tilde{q});I)\Big)
\nonumber
\\
&= 2
\Res\displaylimits_{\tilde{q}=\beta_j} K_j(z_1,\tilde{q})
d_{q}\Big(
K_j(z_2,q)
\Big(U(q,\tilde{q};I)
+U(\sigma_j(q),\sigma_j(\tilde{q});I)\Big)
\Big)\Big|_{q=\tilde{q}}\;,
\nonumber
\\
&\text{where}\quad
U(q,\tilde{q};I)
=\sum_{\substack{I_1\sqcup I_2=I\\I_1,I_2\neq \emptyset}}
\omega^{(0)}_{|I_1|+1}(q, I_1)\omega^{(0)}_{|I_2|+1}(\tilde{q},I_2)\;.
\nonumber
\end{align*}
We evaluate the $d_{q}$ differential and use
$d_{q}U(q,\tilde{q};I)\big|_{q=\tilde{q}}=
\frac{1}{2}d_{\tilde{q}}U(\tilde{q},\tilde{q};I)$.
Integration by parts of this term yields (renaming $\tilde{q}\mapsto q$)
\begin{align*}
\Delta_j(z_1,z_2,I) &= 
\Res\displaylimits_{q=\beta_j} \big(K_j(z_1,q) d_{q}K_j(z_2,q)
-K_j(z_2,q) d_{q}K_j(z_1,q)\big)
\nonumber
\\[-1ex]
&\qquad\qquad
\times
\Big(U(q,q;I)
+U(\sigma_j(q),\sigma_j(q);I)\Big)\;.
\end{align*}
We insert the recursion kernel (\ref{eq:kernel}). There are only contributions if the
differential hits the numerator:
\begin{align}
\Delta_j(z_1,z_2,I) &= 
\Res\displaylimits_{q=\beta_j} \frac{
\big\{\frac{1}{2}(B(z_2,q)-B(z_2,\sigma_i(q)))
\frac{1}{2} \int_{\sigma_i(q)}^q B(z_1,\,.\,)
-z_1\leftrightarrow z_2\big\}
}{(y(q)-y(\sigma_j(q)))^2}
\nonumber
\\
&\qquad\qquad
\times
\Big(\frac{U(q,q;I)}{dx(q) dx(q)}
+\frac{U(\sigma_j(q),\sigma_j(q);I)}{dx(\sigma_j(q))dx(\sigma_j(q))}
\Big)
\nonumber
\\
&= 
\Res\displaylimits_{q=\beta_j} \frac{
B(z_2,q) \frac{1}{2}\int_{\sigma_j(q)}^q B(z_1,\,.\,)
-B(z_1,q) \frac{1}{2}\int_{\sigma_j(q)}^q B(z_2,\,.\,)
}{(y(q)-y(\sigma_i(q)))^2}
\nonumber
\\
&\qquad\qquad
\times
\Big(\frac{U(q,q;I)}{dx(q) dx(q)}
+\frac{U(\sigma_j(q),\sigma_j(q);I)}{dx(\sigma_j(q))dx(\sigma_j(q))}
\Big)\;,
\nonumber
\end{align}
where again the invariance (\ref{inv-Galois}) under 
Galois conjugation was used.

Now we use the Lemma A.1 inside the proof of \cite[Thm 4.6]{Eynard:2007kz}
which states that for a function $f(q)$ which is 
holomorphic at a ramification point $\beta_j$
and Galois-invariant $f(q)=f(\sigma_j(q))$,
\begin{align*}
& \Res\displaylimits_{q=\beta_j} 
\frac{\frac{1}{2} B(z_2,q)\int_{\sigma_j(q)}^{q} B(z_1,\,.\,)
  -\frac{1}{2} B(z_1,q)\int_{\sigma_j(q)}^{q} B(z_1,\,.\,)
}{(y(q)-y(\sigma_j(q)))^2}
f(q)
\\
&= \Res\displaylimits_{q=\beta_j}  \frac{\frac{1}{2} B(z_2,q)
  B(z_1,q)-\frac{1}{2} B(z_1,q)B(z_2,q)}{
(y(q)-y(\sigma_j(q))) dy(q)} f(q)=0
\end{align*}
by l'Hospital's rule.
By the quadratic and linear loop equations
\cite[Prop.\ 2.10]{Hock:2021tbl}
and \cite[Prop.\ 2.6]{Hock:2021tbl}, such a function is
\[
f(q)=
2 (y(\sigma_j(q))-y(q))\frac{\omega^{(0)}_{|I|+1}(\sigma_j(q),I)}{
  dx(\sigma_j(q))}
+
\frac{U(q,q;I)}{dx(q)dx(q)}+\frac{U(\sigma_j(q),\sigma_j(q);I)}{
  dx(\sigma_j(q))
  dx(\sigma_j(q))}\;.
\]
Therefore, we arrive at
\begin{align}
\Delta_j(z_1,z_2,I) &=2
\Res\displaylimits_{q=\beta_j} 
\big(K_j(z_1,q)B(z_2,q) 
-K_j(z_2,q)B(z_1,q) \big) \omega^{(0)}_{|I|+1}(\sigma(q),I)\;.
\label{Delta}
\end{align}
The result (\ref{Delta}) is inserted back into
(\ref{om0I-1beta-diff}). It cancels there the line (***) and removes
(here we rename $q\mapsto \tilde{q}$) in the last term of the line
(**) the Bergman kernel $B(z_2,\tilde{q})$ from
$\omega^{(0)}_2(z_2,\tilde{q})$.  The remainder (\ref{om02}) can be
written as a residue,
\[
  -\omega^{(0)}_2(z_2,\tilde{q})+B(z_2,\tilde{q})= B(z_2,\iota \tilde{q})
  =- \Res\displaylimits_{q=\iota \tilde{q}} \frac{dz_2}{z_2-q}
  \omega^{(0)}_2(q,\tilde{q})\;,
\]  
and this adds to the sum in the third line of (\ref{om0I-1beta-diff})
the missing case $I_3=I$. In summary, we get 
\begin{align}
  &\mathcal{P}_{z_1;\beta_j} \omega^{(0)}_{I|+2}(z_2,z_1,I)
  -\mathcal{P}_{z_1;\beta_j} \omega^{(0)}_{|I|+2}(z_1,z_2,I)
  \label{om0I-1beta-diff-a}
  \\
&  = - 2 \Res\displaylimits_{\tilde{q}=\beta_j}K_j(z_1,\tilde{q})
\sum_{\substack{I_3\subset I \\ I_3\neq \emptyset}}
\omega^{(0)}_{|I_3|+1}(\sigma_j(\tilde{q}),I_3)
\Res\displaylimits_{q=\tilde{q}} \frac{dz_2}{z_2-\iota q}
\omega^{(0)}_{|I\setminus I_3|+2}(\tilde{q},\iota q, I\setminus I_3) 
\tag{\#}
\\
&+\Res\displaylimits_{\tilde{q}=\beta_j}
\frac{dz_1}{z_1-\tilde{q}}
\Res\displaylimits_{q=\iota \tilde{q}} \frac{dz_2}{z_2-q}
\omega^{(0)}_{|I|+2}(q,\tilde{q},I)
\nonumber
\end{align}
To obtain the line (\#) we have used the symmetry 
$\omega^{(0)}_{|I\setminus I_3|+2}(q,\tilde{q},I\setminus I_3)
=\omega^{(0)}_{|I\setminus I_3|+2}(\tilde{q},q,I\setminus I_3)$, allowed
because $|I_3|\geq 1$,  renamed $q\mapsto \iota q$ to prepare the next step
and changed $\Res_{\iota q=\iota\tilde{q}}$ to $\Res_{q=\tilde{q}}$ according
to (\ref{change-iota}). Next, in the line (\ref{om0I-1beta-diff-a}\#), we commute the 
residues with (\ref{commut-residues-beta-0})
for $q\leftrightarrow \tilde{q}$ and take into account that by 
(\ref{no-res-iotabeta}) an inner residue
$\Res_{q=\beta_j}$ does not contribute. The line (\#) thus takes the form
\begin{align*}
\text{(\ref{om0I-1beta-diff-a}\#)}
&  = - 2 \Res\displaylimits_{q=\beta_j}
\frac{dz_2}{z_2-\iota q}
\Res\displaylimits_{\tilde{q}=\beta_j}
K_j(z_1,\tilde{q})
\sum_{\substack{I_1\sqcup I_2=I \\ I_2\neq \emptyset}}
\omega^{(0)}_{|I_1|+2}(\tilde{q}, \iota q,I_1)
\omega^{(0)}_{|I_2|+1}(\sigma_j(\tilde{q}),I_2)
\\
&=- \Res\displaylimits_{q=\beta_j}
\frac{dz_2}{z_2-\iota q}
\Res\displaylimits_{\tilde{q}=\beta_j}
\frac{dz_1}{z_1-\tilde{q}}
\omega^{(0)}_{|I|+2}(\tilde{q},\iota q,I)\;.
\end{align*}
We have taken into account that the inner residue
$\Res\displaylimits_{\tilde{q}=\beta_j}$, which is the case $i=j$ of the first line of
(\ref{sol:omega}), is precisely the projection to the principal part of the
Laurent series of $\omega^{(0)}_{|I|+2}(z_1,\iota q,I)$ at $z_1=\beta_j$.
 Next we use the commutation
rule (\ref{commut-residues-beta-0}) in the other order, taking into
account that by (\ref{no-res-iotabeta}) an inner residue $\Res_{q=\beta_j}$ does
not contribute, and obtain after flip $q\mapsto \iota q$
and reexpession via (\ref{change-iota}) the result
\begin{align*}
\text{(\ref{om0I-1beta-diff-a}\#)}
&=- \Res\displaylimits_{\tilde{q}=\beta_j}
\frac{dz_1}{z_1-\tilde{q}}
\Res\displaylimits_{q= \iota \tilde{q}}
\frac{dz_2}{z_2-q}
\omega^{(0)}_{|I|+2}(\tilde{q},q,I)\;.
\end{align*}
This reduces (\ref{om0I-1beta-diff-a}) to 
\begin{align*}
  &\mathcal{P}_{z_1;\beta_j} \omega^{(0)}_{I|+2}(z_2,z_1,I)
  -\mathcal{P}_{z_1;\beta_j} \omega^{(0)}_{|I|+2}(z_1,z_2,I)
\\
&=\Res\displaylimits_{\tilde{q}=\beta_j}
\frac{dz_1}{z_1-\tilde{q}}
\Res\displaylimits_{q= \iota \tilde{q}} \frac{dz_2}{z_2-q}
\big(\omega^{(0)}_{|I|+2}(q,\tilde{q},I)-\omega^{(0)}_{|I|+2}(\tilde{q},q,I)\big)\;.
\nonumber
\end{align*}
The inner residue $\Res\displaylimits_{q= \iota \tilde{q}}$ vanishes
as proved in the previous section when setting there 
$z_2\mapsto \tilde{q}$ and $z_1\mapsto z_2$.
This completes the proof that
$\omega^{(0)}_{I|+2}(z_1,z_2,I)-\omega^{(0)}_{I|+2}(z_2,z_1,I)$ is holomorphic at 
$z_1=\beta_j$.

\section*{The pole at $z_1= \iota u_k$}

This turns out to be the most involved part. As noted before, the
residue formula directly guarantees symmetry in all arguments except
the first one. We thus compute the pole of
$\omega^{(0)}_{|I|+3}(z_1,u,z_2,I)-\omega^{(0)}_{|I|+3}(z_2,z_1,u,I)$
at $z_1=\iota u$, where $I=\{u_1,...,u_m\}$ with $m\geq 1$.
The representation
(\ref{AHP-pole-uz}) for $z\mapsto \tilde{q}$ allows us to write the
first term as
\begin{align}
&  \mathcal{P}_{z_1;\iota u} \omega^{(0)}_{|I|+3}(z_1,u,z_2,I)
\\
&=  -d_u\Big[\Res\displaylimits_{\tilde{q}=\iota u} \frac{dz_1}{z_1-\tilde{q}}
\sum_{s=1}^{|I|+1}
\sum_{\substack{I_1\sqcup...\sqcup I_s=I\cup\{z_2\}\\I_1,...,I_s\neq \emptyset}}
  \frac{dy(\iota \tilde{q})}{(y(\iota \tilde{q})-y(u))^{s+1}} \prod_{i=1}^s
    \frac{\omega^{(0)}_{|I_i|+1}(u,I_i)}{dx(u)}\Big]\;.
    \nonumber
\end{align}
We have used that $d_u$ commutes with the residue.  Let us denote the
subset which contains $z_2$ by $I_0\cup \{z_2\}$.  This subset can be
located at $s$ places. We always move it to the first position and
take the factor $s$ into account. We take out a $d_{\tilde{q}}$ differential as
in (\ref{AHP-pole-diff}), which precisely cancels this factor
$s$. Integrating by parts and relabelling the subsets $I_i$ gives
\begin{align}
&  \mathcal{P}_{z_1;\iota u} \omega^{(0)}_{|I|+3}(z_1,u,z_2,I)
\label{om0I-1u-a1}
\\
&=  -d_u\Big[\Res\displaylimits_{\tilde{q}=\iota u} \frac{dz_1d\tilde{q}}{(z_1-\tilde{q})^2}
\sum_{s=0}^{|I|}  \sum_{\substack{
    I_0\sqcup I_1\sqcup ...\sqcup I_s=I\\ I_1,...,I_s\neq \emptyset}}
\frac{\omega^{(0)}_{|I_0|+2}(u,z_2,I_0)\prod_{i=1}^s\omega^{(0)}_{|I_i|+1}(u,I_i)}{
  ((y(\iota \tilde{q})-y(u))dx(u))^{s+1}}
  \Big]\;.
    \nonumber
\end{align}
The case $I_0=\emptyset$ is left as it is. For $I_0\neq \emptyset$, 
$\omega^{(0)}_{|I_0|+2}(u,z_2,I_0)$ is of shorter length and symmetric by
induction hypothesis. We represent it by the recursion formula (\ref{sol:omega}) for 
$\omega^{(0)}_{|I_0|+2}(z_2,u,I_0)$:
\begin{align}
&  \mathcal{P}_{z_1;\iota u} \omega^{(0)}_{|I|+3}(z_1,u,z_2,I)
\label{om0I-1u-a2}
\\
&=  -d_u\Big[\Res\displaylimits_{\tilde{q}=\iota u} \frac{dz_1d\tilde{q}}{(z_1-\tilde{q})^2}
\sum_{s=1}^{|I|}
\sum_{\substack{I_1\sqcup ...\sqcup I_s=I\\ I_1,...,I_s\neq \emptyset}}
\frac{\omega^{(0)}_{2}(u,z_2)\prod_{i=1}^s\omega^{(0)}_{|I_i|+1}(u,I_i)}{
    ((y(\iota \tilde{q})-y(u))dx(u))^{s+1}}
  \Big]
  \nonumber
\\
&  -d_u\Big[\Res\displaylimits_{\tilde{q}=\iota u} \frac{dz_1d\tilde{q}}{(z_1-\tilde{q})^2}
\sum_{s=0}^{|I|-1}  \sum_{\substack{
    I_0\sqcup I_1\sqcup ...\sqcup I_s=I\\ I_0,I_1,...,I_s\neq \emptyset}}
\frac{\prod_{i=1}^s\omega^{(0)}_{|I_i|+1}(u,I_i)}{((y(\iota \tilde{q})-y(u))dx(u))^{s+1}}
\tag{\S}
\\
&\qquad\times \Big\{
2\sum_{i} \Res\displaylimits_{q=\beta_i} \;
  K_i(z_2,q)\sum_{\substack{I'\sqcup I''=I_0\\ I''\neq \emptyset}}
\omega^{(0)}_{|I'|+2}(q,u,I')  \omega^{(0)}_{|I''|+1}(\sigma_i(q),I'')
\tag{\#}
\nonumber
\\
&-2 \sum_{k=1}^{|I_0|} d_{u_k} \Big[
\Res\displaylimits_{q=\iota u_k} \tilde{K}(z_2,q,u_k)
d_{u_k}^{-1}\Big(
\sum_{\substack{I'\sqcup I''=I_0\\ I''\neq \emptyset}}
\omega^{(0)}_{|I'|+2}(q,u,I')  \omega^{(0)}_{|I''|+1}(q,I'')\Big)\Big]
\tag{\#\#}
\\
&-2 d_{u} \Big[\Res\displaylimits_{q= \iota u} \tilde{K}(z_2,q,u)
d_{u}^{-1}
\Big(
\sum_{\substack{I'\sqcup I''=I_0\\ I''\neq \emptyset}}
\omega^{(0)}_{|I'|+2}(q,u,I')  \omega^{(0)}_{|I''|+1}(q,I'')\Big)\Big]
\Big\}
  \Big]\;.
    \nonumber
\end{align}

We return later to (\ref{om0I-1u-a2}) and continue here 
with $\mathcal{P}_{z_1,\iota u}\omega^{(0)}_{|I|+3}(z_2,z_1,u,I)$.
This is a residue applied to the recursion formula (\ref{sol:omega})
for $z\mapsto z_2$ and $I\mapsto \{\tilde{q},u\}\cup I$.
In the cases where independent residues
commute, $(\tilde{q},u)$ must be located in the same
$\omega^{(0)}$ to produce such a pole:
\begin{align}
&\mathcal{P}_{z_1;\iota u}\omega^{(0)}_{|I|+3}(z_2,z_1,u,I)
\label{om0I-1u-b}
\\
&= \Res\displaylimits_{\tilde{q}= \iota u} \frac{dz_1}{z_1-\tilde{q}}\Big\{
2\sum_i \Res\displaylimits_{q=\beta_i} K_i(z_2,q)
\sum_{\substack{I'\sqcup I''=I\\I''\neq \emptyset}}
\omega^{(0)}_{|I'|+3}(q,\tilde{q},u,I')
\omega^{(0)}_{|I''|+1}(\sigma_i(q),I'')
\tag{\dag}
\\[-1ex]
&-2\sum_{k=1}^{|I|} d_{u_k}\Big[\Res\displaylimits_{q= \iota u_k} \tilde{K}(z_2,q,u_k)
\sum_{\substack{I'\sqcup I''=I\\I''\neq \emptyset}}
d_{u_k}^{-1}\Big(\omega^{(0)}_{|I'|+3}(q,\tilde{q},u,I')
\omega^{(0)}_{|I''|+1}(q,I'')\Big)\Big]
\tag{\dag\dag}
\\
&-2 d_{\tilde{q}}\Big[\Res\displaylimits_{q= \iota \tilde{q}} \tilde{K}(z_2,q,\tilde{q})
\sum_{\substack{I'\sqcup I''=I\\I''\neq \emptyset}}
d_{\tilde{q}}^{-1}\Big(\omega^{(0)}_{|I'|+3}(q,\tilde{q},u,I')
\omega^{(0)}_{|I''|+1}(q,I'')\Big)\Big]
\tag{\ddag'}
\\
&-2 d_{\tilde{q}}\Big[\Res\displaylimits_{q= \iota \tilde{q}} \tilde{K}(z_2,q,\tilde{q})
\sum_{\substack{I'\sqcup I''=I}}
d_{\tilde{q}}^{-1}\Big(\omega^{(0)}_{|I'|+2}(q,\tilde{q},I')
\omega^{(0)}_{|I''|+2}(q,u,I'')\Big)\Big]
\tag{\ddag''}
\\
&-2 d_{u}\Big[\Res\displaylimits_{q= \iota u} \tilde{K}(z_2,q,u)
\sum_{\substack{I'\sqcup I''=I\\I''\neq \emptyset}}
d_{u}^{-1}\Big(\omega^{(0)}_{|I'|+3}(q,\tilde{q},u,I')
\omega^{(0)}_{|I''|+1}(q,I'')\Big)\Big]
\tag{\ddag\ddag'}
\\
&-2 d_{u}\Big[\Res\displaylimits_{q= \iota u} \tilde{K}(z_2,q,u)
\sum_{\substack{I'\sqcup I''=I}}
d_{u}^{-1}\Big(\omega^{(0)}_{|I'|+2}(q,\tilde{q},I')
\omega^{(0)}_{|I''|+2}(q,u,I'')\Big)\Big]\Big\}\;.
\tag{\ddag\ddag''}
\end{align}
In the line (\dag), because of shorter length, we have 
$\omega^{(0)}_{|I'|+3}(q,\tilde{q},u,I')=\omega^{(0)}_{|I'|+3}(\tilde{q},u,q,I')$
by induction hypothesis. We commute the residues and use
the identity (\ref{om0I-1u-a1}) for $z_2\mapsto q$ and $I\mapsto I'$.
After relabelling the index sets and commuting the residues again we
get for the line (\dag) 
exactly the same expression as the line (\#), multiplied by the line (\S), of
(\ref{om0I-1u-a2}).  Similar considerations prove that the line
(\dag\dag) of (\ref{om0I-1u-b}) is equal to the line (\#\#), multiplied by the
line (\S), of (\ref{om0I-1u-a2}).

The lines (\ddag') and (\ddag'') can be combined to a sum over
decompositions $I' \sqcup I''=I\cup \{u\}$ of
$\omega^{(0)}_{|I'|+3}(q,\tilde{q},I')\omega^{(0)}_{|I''|+1}(q,I'')$;
and the lines (\ddag\ddag') and (\ddag\ddag'') to a sum
over decompositions $I' \sqcup I''= I\cup \{\tilde{q}\}$ of
$\omega^{(0)}_{|I'|+3}(q,u,I')\omega^{(0)}_{|I''|+1}(q,I'')$.

With these considerations
and $\omega^{(0)}_{|I|+3}(z_1,z_2,u,I)\equiv \omega^{(0)}_{|I|+3}(z_1,u,z_2,I)$
the difference
between (\ref{om0I-1u-a2}) and (\ref{om0I-1u-b}) reduces to
\begin{align}
  &  \mathcal{P}_{z_1;\iota u} \big(\omega^{(0)}_{|I|+3}(z_1,z_2,u,I)
  -\omega^{(0)}_{|I|+3}(z_2,z_1,u,I)\big)
\label{om0I-1u-diff}
\\
&=  -d_u\Big[\Res\displaylimits_{\tilde{q}=\iota u} \frac{dz_1d\tilde{q}}{(z_1-\tilde{q})^2}
\sum_{s=1}^{|I|}
\sum_{\substack{I_1\sqcup ...\sqcup I_s=I\\ I_1,...,I_s\neq \emptyset}}
\frac{\omega^{(0)}_{2}(u,z_2)\prod_{i=1}^s\omega^{(0)}_{|I_i|+1}(u,I_i)}{
    ((y(\iota \tilde{q})-y(u))dx(u))^{s+1}}
  \Big]
\tag{a}
\\
&  +d_u\Big[\Res\displaylimits_{\tilde{q}= \iota u} \frac{dz_1d\tilde{q}}{(z_1-\tilde{q})^2}
\sum_{s=0}^{|I|-1}  \sum_{\substack{
    I_0\sqcup I_1\sqcup ...\sqcup I_s=I\\ I_0,I_1,...,I_s\neq \emptyset}}
\frac{\prod_{i=1}^s\omega^{(0)}_{|I_i|+1}(u,I_i)}{((y(\iota\tilde{q})-y(u))dx(u))^{s+1}}
\tag{b'}
\\
&\qquad\times 
2 d_{u} \Big[
\Res\displaylimits_{q= \iota u} \tilde{K}(z_2,q,u)
d_{u}^{-1}
\Big(
\sum_{\substack{I'\sqcup I''=I_0\\ I''\neq \emptyset}}
\omega^{(0)}_{|I'|+2}(q,u,I')  \omega^{(0)}_{|I''|+1}(q,I'')\Big)
  \Big]\Big]
\tag{b''}
    \\
&-2\Res\displaylimits_{\tilde{q}= \iota u} \frac{dz_1d\tilde{q}}{(z_1-\tilde{q})^2}
\Res\displaylimits_{q= \iota \tilde{q}} \tilde{K}(z_2,q,\tilde{q})
\sum_{\substack{I'\sqcup I''=I\cup \{u\} \\I''\neq \emptyset}} \!\!\! 
d_{\tilde{q}}^{-1}\big(\omega^{(0)}_{|I'|+2}(q,\tilde{q},I')\big)
\omega^{(0)}_{|I''|+1}(q,I'')
\tag{c}
\\
&+2 d_{u}\Big[
\Res\displaylimits_{\tilde{q}=\iota u} \frac{dz_1}{z_1-\tilde{q}}
\Res\displaylimits_{q= \iota u} \tilde{K}(z_2,q,u) \!\!\!\!
\sum_{\substack{I'\sqcup I''=I\cup \{\tilde{q}\} \\I''\neq \emptyset}} \!\!\!\!\!\!
d_{u}^{-1}\big(\omega^{(0)}_{|I'|+2}(q,u,I')\big)
\omega^{(0)}_{|I''|+1}(q,I'')\Big]\,.
\tag{d}
\end{align}
In the last two lines we have either integrated by parts in $\tilde{q}$ or moved
$d_u$ though the outer residue.

We process the individual lines of (\ref{om0I-1u-diff}), where the
main tool will be the commutation rule of residues (\ref{commut-res-u}).
In (\ref{om0I-1u-diff}c), because of shorter length,
the induction hypothesis gives $\omega^{(0)}_{|I'|+2}(q,\tilde{q},I')=
\omega^{(0)}_{|I'|+2}(\tilde{q},q,I')$. Lemma~\ref{lemma:resK} for
$I\mapsto I\cup \{u\}$ gives 
\begin{align*}
\text{(\ref{om0I-1u-diff}c)}
&=-\Res\displaylimits_{\tilde{q}=\iota u} \frac{dz_1d\tilde{q}}{(z_1-\tilde{q})^2}
\Res\displaylimits_{q=\iota \tilde{q}} \frac{dz_2dq}{(z_2-q)^2}
\sum_{s=1}^{|I|+1}\frac{1}{s}
\sum_{\substack{I_1\sqcup...\sqcup I_s=I\cup \{u\}\\ I_1,...,I_s\neq \emptyset}}
\prod_{i=1}^s 
\frac{\omega^{(0)}_{|I_i|+1}(q,I_i)}{dx(q)(y(\iota \tilde{q})-y(q))}\;.
\end{align*}
We flip $q\mapsto \iota q$, take (\ref{change-iota}) into account and
use the commutation rule (\ref{commut-res-u}) of residues. There is no
contribution if $\Res_{\tilde{q}=\iota u}$ is the inner residue. If
$\Res_{q=\iota u}$ is the inner residue, then the only contribution is
the case that $u$ appears in a factor $\omega^{(0)}_2(\iota
q,u)$. This factor can arise at $s$ places, which compensate the
prefactor $\frac{1}{s}$:
\begin{align*}
\text{(\ref{om0I-1u-diff}c)}
&=
\Res\displaylimits_{\tilde{q}=\iota u}
\Res\displaylimits_{q=\iota u}
\frac{dz_1d\tilde{q}}{(z_1-\tilde{q})^2}
\frac{dz_2d(\iota q)}{(z_2-\iota q)^2}
\frac{\omega^{(0)}_2(\iota q,u)}{dx(\iota q)(y(\iota\tilde{q})-y(\iota q))}
\\
&\times \sum_{s=1}^{|I|}
\sum_{\substack{I_1\sqcup...\sqcup I_s=I\\ I_1,...,I_s\neq \emptyset}}
\prod_{i=1}^s 
\frac{\omega^{(0)}_{|I_i|+1}(\iota q,I_i)}{dx(\iota q)
  (y(\iota \tilde{q})-y(\iota q))}\;.
\end{align*}
With $\omega^{(0)}_2(\iota q,u)=\omega^{(0)}_2(q,\iota u)$ the residue is now
easy to compute:
\begin{align*}
\text{(\ref{om0I-1u-diff}c)} 
&=d_u\Big[
\Res\displaylimits_{\tilde{q}=\iota u}
\frac{dz_1d\tilde{q}}{(z_1-\tilde{q})^2}
\frac{dz_2du}{(z_2-u)^2}
\frac{1}{dx(u) (y(\iota \tilde{q})-y(u))}
\\[-1ex]
&\times \sum_{s=1}^{|I|}
\sum_{\substack{I_1\sqcup...\sqcup I_s=I\\ I_1,...,I_s\neq \emptyset}}
\prod_{i=1}^s 
\frac{\omega^{(0)}_{|I_i|+1}(u,I_i)}{dx(u)(y(\iota\tilde{q})-y(u))}\Big]\;.
\end{align*}
This cancels partly with (\ref{om0I-1u-diff}a). Introducing
\begin{align}
b_r(z_1,u) &:=  \Res\displaylimits_{q=\iota u}
\frac{dz_1dq}{(z_1-q)^2(y(\iota q)-y(u))^{r}}\;,
\label{bz1u}
\\
&\text{where}\qquad 
\frac{dz_1 dq}{(z_1-q)^2 dx(q)}
=-\sum_{r=0}^\infty b_{r+1}(z_1,u) \cdot (y(\iota q)-y(u))^r
\nonumber
\end{align}
near $q=\iota u$ (here $dx(q)=-dy(\iota q)$ is used),
we get with (\ref{om02})
\begin{align}
\text{(\ref{om0I-1u-diff}a+c)} 
&=d_u\Big[
\sum_{s=1}^{|I|} b_{s+1}(z_1,u)
\frac{B(z_2,\iota u)}{dx(u)}
\sum_{\substack{I_1\sqcup...\sqcup I_s=I\\ I_1,...,I_s\neq \emptyset}}
\prod_{i=1}^s \frac{\omega^{(0)}_{|I_i|+1}(u,I_i)}{dx(u)}\Big]\;.
\label{om0I-1u-diffac}
\end{align}

We pass to (\ref{om0I-1u-diff}d) where, because of shorter length, the
induction hypothesis allows us to write
$\omega^{(0)}_{|I'|+2}(q,u,I') =\omega^{(0)}_{|I'|+2}(u,q,I')$.
Lemma~\ref{lemma:resK} for
$\tilde{q}\mapsto u$ and $I\mapsto I\cup \{\tilde{q}\}$ gives
\begin{align*}
  \text{(\ref{om0I-1u-diff}d)}
&= d_{u}\Big[
\Res\displaylimits_{\tilde{q}= \iota u} \frac{dz_1}{z_1-\tilde{q}}
\Res\displaylimits_{q= \iota u} \frac{dz_2dq}{(z_2-q)^2}
\sum_{s=1}^{|I|+1}\frac{1}{s}
\sum_{\substack{I_1\sqcup...\sqcup I_s=I\cup \{\tilde{q}\} \\ I_1,...,I_s\neq \emptyset}}
\prod_{i=1}^s 
\frac{\omega^{(0)}_{|I_i|+1}(q,I_i)}{dx(q)(y(\iota u)-y(q))}
\Big]\;.
\end{align*}
We subtract the vanishing residue in the other order
$\Res\displaylimits_{q= \iota u}\Res\displaylimits_{\tilde{q}= \iota u}$ and use
the commutation rule (\ref{commut-res-u}) to have an inner
$\Res\displaylimits_{q=\tilde{q}}$.  The only contribution is the case
of a factor $\omega^{(0)}(q,\tilde{q})$.  It can occur at $s$ places
and cancels the factor $\frac{1}{s}$: 
\begin{align*}
  \text{(\ref{om0I-1u-diff}d)}
&=- d_{u}\Big[
\Res\displaylimits_{\tilde{q}= \iota u} \frac{dz_1}{z_1-\tilde{q}}
\Res\displaylimits_{q=\tilde{q}}\frac{dz_2dq}{(z_2-q)^2}
\\*[-0.5ex]
&\times
\frac{\omega^{(0)}_{2}(q,\tilde{q})}{dx(q)(y(\iota u)-y(q))}
\sum_{s=1}^{|I|} 
\sum_{\substack{I_1\sqcup ... \sqcup I_s=I\\ I_1,...,I_s\neq \emptyset}}
\prod_{i=1}^s
\frac{\omega^{(0)}_{|I_i|+1}(q,I_i)}{dx(q)(y(\iota u)-y(q))}
\Big]
\\
&= d_{u}\Big[
\Res\displaylimits_{\tilde{q}= \iota u} \frac{dz_1d\tilde{q}}{(z_1-\tilde{q})^2}
\frac{dz_2}{(z_2-\tilde{q})^2}
\frac{1}{x'(\tilde{q})(y(\iota u)-y(\tilde{q}))}
\\
&\times
\sum_{s=1}^{|I|} 
\sum_{\substack{I_1\sqcup ... \sqcup I_s=I\\ I_1,...,I_s\neq \emptyset}}
\prod_{i=1}^s
\frac{\omega^{(0)}_{|I_i|+1}(\tilde{q},I_i)}{dx(\tilde{q})(y(\iota u)-y(\tilde{q}))}
\Big]\;.
\end{align*}
Using (\ref{bz1u}) we bring (\ref{om0I-1u-diff}d) into the form
\begin{align*}
  \text{(\ref{om0I-1u-diff}d)}
&= -d_{u}\Big[
\sum_{s=1}^{|I|} \sum_{r=0}^{s}b_{s-r+1} (z_1,u)
\Res\displaylimits_{\tilde{q}=\iota u} 
\frac{dz_2d\tilde{q}}{(z_2-\tilde{q})^2}
\frac{1}{(y(\iota u)-y(\tilde{q}))^{r+1}}
\\[-1ex]
&\times
\sum_{\substack{I_1\sqcup ... \sqcup I_s=I\\ I_1,...,I_s\neq \emptyset}}
\prod_{i=1}^r
\frac{\omega^{(0)}_{|I_i|+1}(\tilde{q},I_i)}{dx(\tilde{q})}
\prod_{i=r+1}^s
\frac{(y(\iota \tilde{q})-y(u))}{(y(\iota u)-y(\tilde{q}))}
\frac{\omega^{(0)}_{|I_i|+1}(\tilde{q},I_i)}{dx(\tilde{q})}
\Big]\;.
\end{align*}

In the line (\ref{om0I-1u-diff}b''), because of shorter length,
the induction hypothesis implies
$\omega^{(0)}_{|I'|+2}(q,u,I')=\omega^{(0)}_{|I'|+2}(u,q,I')$.
Lemma~\ref{lemma:resK} for
$I\mapsto I_0$ and $\tilde{q}\mapsto u$ gives
\begin{align*}
\text{(\ref{om0I-1u-diff}b'')}
&=d_{u} \Big[
\Res\displaylimits_{q= \iota u} \frac{dz_2dq}{(z_2-q)^2}
\sum_{s=1}^{|I|}\frac{1}{s}
\sum_{\substack{I_1\sqcup...\sqcup I_r=I_0\\ I_1,...,I_r\neq \emptyset}}
\prod_{i=1}^r 
\frac{\omega^{(0)}_{|I_i|+1}(q,I_i)}{dx(q)(y(\iota u)-y(q))}\Big]\;.
\end{align*}
We decide to carry out the
$d_u$ differential to remove the factor $\frac{1}{s}$.
We multiply it with the line (\ref{om0I-1u-diff}b') in which
(\ref{bz1u}) is used to simplify the expression. We can cancel 
$dy(\iota u)=-dx(u)$. After relabelling summation indices and index sets we get for
(\ref{om0I-1u-diff}b'$\times$b'')
\begin{align*}
  \text{(\ref{om0I-1u-diff}b'$\times$b'')}
  &= d_u\Big[\sum_{s=1}^{I} \sum_{r=1}^s b_{s-r+1}(z_1,u)
  \sum_{\substack{I_1\sqcup ...\sqcup I_s=I\\I_1,...,I_s \neq \emptyset}}
  \Res\displaylimits_{q=\iota u}\Big(
  \frac{dz_2dq}{(z_2-q)^2(y(\iota u)-y(q))^{r+1}} 
  \\
  &\times \prod_{i=1}^r \frac{\omega^{(0)}_{|I_i|+1}(q,I_i)}{dx(q)}\Big)
\prod_{i=r+1}^s \frac{\omega^{(0)}_{|I_i|+1}(u,I_i)}{dx(u)}\Big]\;.
\end{align*}
We observe that the sum (\ref{om0I-1u-diff}a+c) of lines, given in
(\ref{om0I-1u-diffac}),
is with $dx(u)=-dy(\iota u)$ just the missing case $r=0$ of 
(\ref{om0I-1u-diff}b'$\times$b'').

We collect the evaluated lines of (\ref{om0I-1u-diff}):
\begin{align}
  &  \mathcal{P}_{z_1;\iota u} \big(\omega^{(0)}_{|I|+3}(z_1,z_2,u,I)
  -\omega^{(0)}_{|I|+3}(z_2,z_1,u,I)\big)
\label{om0I-1u-diff-eval}
\\
&= d_u\Big[\sum_{s=1}^{|I|} \sum_{r=0}^s b_{s-r+1}(z_1,u)
  \sum_{\substack{I_1\sqcup ...\sqcup I_s=I\\I_1,...,I_s \neq \emptyset}}
  \Res\displaylimits_{q=\iota u}\Big\{
  \frac{dz_2dq}{(z_2-q)^2(y(\iota u)-y(q))^{r+1}} 
    \nonumber
  \\
  &\times \prod_{i=1}^r \frac{\omega^{(0)}_{|I_i|+1}(q,I_i)}{dx(q)}
\Big(
  \prod_{i=r+1}^s \frac{\omega^{(0)}_{|I_i|+1}(u,I_i)}{dx(u)}
-
\prod_{i=r+1}^s
\frac{(y(\iota q)-y(u))}{(y(\iota u)-y(q))}
\frac{\omega^{(0)}_{|I_i|+1}(q,I_i)}{dx(q)}\Big)\Big\}
\Big]\;.
\nonumber
\end{align}
In fact the sum can be restricted to $0\leq r<s$.

Similarly to \cite[eq.\ (2.7)]{Hock:2021tbl} we introduce
the representation
near $q= \iota u$ (setting there 
$q\mapsto \iota u$ and $z\mapsto q$ and taking $x(\iota q)=-y(q)$ into account)
\begin{align}
&  \frac{(y(\iota q)-y(u))}{(y(\iota u)-y(q))}
  \frac{\omega^{(0)}_{|I|+1}(q,I)}{dx(q)}
  =\sum_{n=0}^\infty (y(\iota q)-y(u))^n \nabla^n\omega^{(0)}_{|I|+1}(\iota u,I)\;,
\label{omega-Taylor}
\\[-1ex]
&\text{where} \quad
\nabla^n\omega^{(0)}_{|I|+1}(\iota u ,I):=
  \Res\displaylimits_{q\to \iota u}
  \frac{\omega^{(0)}_{|I|+1}(q,I)}{(y(q)-y(\iota u))(y(\iota q)-y(u))^n}\;.
  \nonumber
\nonumber
\end{align}
In terms of
\begin{align}
a_r(z_2,u):= \Res\displaylimits_{q= \iota u}
  \frac{dz_2dq}{(z_2-q)^2(y(\iota u)-y(q))(y(\iota q)-y(u))^{r}} 
\end{align}
we thus get
\begin{align}
  &  \mathcal{P}_{z_1; \iota u} \big(\omega^{(0)}_{|I|+3}(z_1,z_2,u,I)
  -\omega^{(0)}_{|I|+3}(z_2,z_1,u,I)\big)
\label{P12-final}
  \\
&= d_u\Big[\sum_{s=1}^{|I|} \sum_{r=0}^{s-1} \sum_{t=0}^r
b_{s-r+1}(z_1,u)a_{r-t}(z_2,u)
\nonumber
  \\
  &\times   \sum_{\substack{I_1\sqcup ...\sqcup I_s=I\\I_1,...,I_s \neq \emptyset}}
\Big(\sum_{n_1+...,n_r=t}
\prod_{i=1}^r \nabla^{n_i}\omega^{(0)}_{|I_i|+1}( \iota u,I_i)
\prod_{i=r+1}^s \frac{\omega^{(0)}_{|I_i|+1}(u,I_i)}{dx(u)}  
\nonumber
\\
&\qquad-\sum_{n_1+...,n_s=t}
  \prod_{i=1}^s \nabla^{n_i}\omega^{(0)}_{|I_i|+1}(\iota u,I_i)
\Big)\Big]
\nonumber
\\
&=
d_u\Big[\sum_{l=1}^{|I|} \sum_{k=0}^{|I|-l} 
b_{l+1}(z_1,u)a_{k}(z_2,u)
\nonumber
  \\
  &\times  \Big(
  \sum_{r=k+l}^{|I|}
  \sum_{\substack{J_1\sqcup ...\sqcup J_r=I\\J_1,...,J_r \neq \emptyset}}
\sum_{\substack{m_1+...,m_{r-l}\\ ~~~=r-k-l}}
\prod_{i=1}^{r-l} \nabla^{m_i}\omega^{(0)}_{|J_i|+1}(\iota u,J_i)
\prod_{i=r-l+1}^r \frac{\omega^{(0)}_{|J_i|+1}(u,J_i)}{dx(u)}  
\tag{*}
\\
&\qquad-\sum_{s=k+l}^{|I|}
  \sum_{\substack{I_1\sqcup ...\sqcup I_s=I\\I_1,...,I_s \neq \emptyset}}
  \sum_{\substack{n_1+...,n_s\\ ~~~=s-k-l}}
  \prod_{i=1}^s \nabla^{n_i}\omega^{(0)}_{|I_i|+1}(\iota u,I_i)
  \Big)\Big]\;. \tag{**}
\end{align}

The remaining task is to prove that for any pair $(k,l)$ the
difference in the last two lines (*) and (**) vanishes
identically. If this succeeds we have shown that
$\omega^{(0)}_{|I|+3}(z_1,z_2,u,I)
-\omega^{(0)}_{|I|+3}(z_2,z_1,u,I)$ is holomorphic at any potential pole and thus
holomorphic everywhere on $\mathbb{P}^1$, i.e.\ equal to zero.

\section{Translation into a combinatorial problem}

\label{sec:conjecture}

In this section, we prove Theorem~\ref{thm:symmetry-differentials} by reduction to the combinatorial statement given in Theorem~\ref{conj:main}. The proof of the latter theorem is given in the next section. 

\begin{proof}[Proof of Theorem~\ref{thm:symmetry-differentials}]
We use \cite[Lemma 2.2]{Hock:2021tbl}
for $q\mapsto \iota q$ which in the changed convention (see remarks after
Thm.~\ref{thm:flip})  states that, 
as consequence of the involution identity (\ref{eq:flip-om})
and with (\ref{omega-Taylor}), one has
\begin{align}
\label{involution-expand}
  \frac{\omega^{(0)}_{|I|+1}(u,I)}{dx(u)}
  =\sum_{r=1}^{|I|}\frac{1}{r}
  \sum_{\substack{I_1\sqcup ...\sqcup I_r=I \\I_1,...,I_r\neq \emptyset}} ~
  \sum_{n_1+...+n_r=r-1} \nabla^{n_i}\omega^{(0)}_{|I_i|+1}(\iota u,I_i)\;.
\end{align}
Inserting this relation into the line (*) of (\ref{P12-final}) one
generates precisely the terms of the same structure as in the last
line (**). But it is by no means clear that all prefactors cancel! It
is the consequence of an intriguing combinatorial identity that we
will now extract from the problem.

Let us fix integers $(k,l)$ with $l\geq 1$ and $k\geq 0$, an
integer $s$ with $k+l\leq s\leq |I|$ and a partition
$I_1\sqcup ....\sqcup I_s=I\equiv \{u_1,...,u_m\}$ into non-empty
disjoint $I_i$. We can compare any two $I_i,I_j$ by
$I_i<I_j$ if $\min(a\;:~u_a\in I_i)<\min(b\;:~u_b\in I_j)$. Then the
sum over decompositions $I=I_1\sqcup ....\sqcup I_s$ in (**) can be
collected into a global factor $s!$ times the sum over ordered
partitions $\displaystyle J_1 \stackrel{<}{\sqcup} ...\stackrel{<}{\sqcup} J_s$
with $J_1<J_2<...<J_s$. We thus
consider in the last line (**) of (\ref{P12-final})
a single term 
\begin{align}
  s! \nabla^{n_1}\omega^{(0)}_{|I_1|+1}(\iota u,J_1)\cdots
   \nabla^{n_s}\omega^{(0)}_{|I_s|+1}(\iota u,J_s)
  \label{omega-part-1}
\end{align}
with $J_1<J_2<...<J_s$ and $n_1+...+n_s=s-k-l$, and keep these data fixed. 

Let $\mathsf{P}_k(n)$ be the set of partitions of $n$ into $k$ parts.
To the tuple $(n_1,...,n_s)$ we assign 
a partition $\nu \in \mathsf{P}_s(2s-k-l)$ by declaring $\nu$ to be the
weakly decreasingly ordered $\{n_1+1,n_2+1,...,n_s+1\}$. We say that
$(n_1,...,n_s)$ is of type $\nu$.

Similarly, we write a term contributing to the line (*) of (\ref{P12-final}) as
\begin{align}
&  (r-l)! \label{omega-part-2}
\nabla^{m_1}\omega^{(0)}_{|J_1'|+1}(\iota u,J'_1)\cdots
 \nabla^{m_{r-l}} \omega^{(0)}_{|J_{r-l}'|+1}(\iota u,J'_{r-l})
 \frac{\omega^{(0)}_{|I'_{1}|+1}(u,I'_1)}{dx(u)}\cdots
 \frac{\omega^{(0)}_{|I'_{l}|+1}(u,I'_l)}{dx(u)}
\end{align}
where $J_1'<J'_2<...<J'_{r-l}$ are ordered, $I_1',....,I_{l}'$ remain
unordered and $m_1+...+m_{r-l}=r-k-l$.  We also assign to the tuple
$(m_1,...,m_{r-l})$ a partition $\mu \in \mathsf{P}_{r-l}(2r-k-2l)$,
i.e.\ the weakly decreasingly ordered $\{m_1+1,...,m_{r-l}+1\}$.

Now the question is: Does a term (\ref{omega-part-2}), after expanding
the final $l$ factors via (\ref{involution-expand}), produce the term
of the form (\ref{omega-part-1}), and with which fraction?  The
expansion of (\ref{omega-part-2}) will also produce other terms which
are counted for other cases of (\ref{omega-part-1}); here they can
safely be discarded. We sum the fractions contributing to
(\ref{omega-part-1}) over \emph{all} terms (\ref{omega-part-2}). To
have symmetry of the meromorphic differentials, the sum of fractions
must be $s!$. 

Clearly, because the first $r-l$ factors in (\ref{omega-part-2}) stay intact, we need
\begin{align}
  \{(m_1,J'_1)  ,....,(m_{r-l},J'_{r-l}) \} \subset
  \{(n_1,J_1)  ,....,(n_s,J_s) \} \;.
\label{subsetcondition}
\end{align}
The complement must be produced by expanding the last $l$ factors of
(\ref{omega-part-2}). Let us split our task and
fix an integer $k+l\leq r\leq s$ and a partition
$\mu\in \mathsf{P}_{r-l}(2r-k-2l)$ and ask how many subsets
(\ref{subsetcondition}) exist under the condition that
$(m_1,...,m_{r-l})$ is of type $\mu$. Let $g_1$ of the $m_i$ be equal
to $0$ and $h_1$ of the $n_i$ be equal to $0$. We are compatible with
the subset condition (\ref{subsetcondition}) for every choice of
$h_1-g_1$ elements $(0,J_*)$ which are removed (they go into the
complement). The number of these choices is $\binom{h_1}{g_1}$. Next,
let $g_2$ of the $m_i$ be equal to $1$ and $h_2$ of the $n_i$ be equal
to $1$. We are compatible with the subset condition
(\ref{subsetcondition}) for every choice of $h_2-g_2$ elements
$(1,J_*)$ that are removed (they go into the complement). The number
of these choices is $\binom{h_2}{g_2}$. Continuing, we have
\[
  \binom{\nu}{\mu}:= \lim_{N\to \infty} \prod_{i=1}^N
  \binom{h_i}{g_i}
\]
possible subsets (\ref{subsetcondition}) if 
$(m_1,...,m_{r-l})$ is of type $\mu$ and
$(n_1,...,n_{r-l})$ of type $\nu$ and 
\[
  \nu=(\dots, \underbrace{2,...,2}_{h_2},\underbrace{1,...,1}_{h_1})\;,\qquad
    \mu=(\dots, \underbrace{2,...,2}_{g_2},\underbrace{1,...,1}_{g_1})\;.
\]
Note that $g_i=h_i=0$ for large enough $i$ so that the product
$\prod_{i=1}^N$ becomes constant for $N$ large enough. We stress that
it suffices to know the partition type of the tuple
$(m_1,...,m_{r-l})$ because the subset
$\{(m_1,J'_1) ,....,(m_{r-l},J'_{r-l}) \}$ is obtained by removal of
elements from $\{(n_1,J_1) ,....,(n_s,J_s) \}$. The remainder inherits its
order. In particular, one knows the bijection between $J'_j$ and a
subset of $J_i$, and the concrete exponent $m_j$ is completely
determined by the choice of removed elements. In particular, for given
choice of removed elements, there is only one tuple
$((m_1,J'_1) ,....,(m_{r-l},J'_{r-l}) )$ for which $(m_1,...,m_{r-l})$
is of type $\mu$.

We now study the complement
\begin{align*}
&\{(m''_1,J''_1)  ,....,(m''_{s-r+l},J_{s-r+l}'')\}
\\
&:=  \{(n_1,J_1)  ,....,(n_s,J_s) \} \setminus
\{(m_1,J_1')  ,....,(m_{r-l},J'_{r-l}) \}\;.
\end{align*}
It consists of $s-r+l$ elements and must result from the expansion of
$l$ factors $\frac{1}{dx(u)} \omega^{(0)}_{|I_*'|+1}(u,I'_*)$ according to
(\ref{involution-expand}).  The exponent tuple $(m_1'',...,m_{s-r+l}'')$
of the complement will be of type $\rho$ for a partition
$\rho:=\nu\setminus \mu \in \mathsf{P}_{s-r+l}(2s-2r+l)$.

Let us specify an integer $p_1\geq 1$ and look at the expansion
(\ref{involution-expand}) of the first factor
$\frac{1}{dx(u)}\omega^{(0)}_{|I'_{1}|+1}(u,I'_1)$ in
(\ref{omega-part-2}) into $p_1$ terms
$\nabla^{h_1}\omega^{(0)}_{|J_{1}'''|+1}(\iota u,J'''_{1})\cdots
\nabla^{h_{p_1}}\omega^{(0)}_{|J'''_{p_1}|+1}(\iota u,J'''_{p_1})$
(which requires $|I'_1|\geq p_1$).  Such an expansion can contribute
to (\ref{omega-part-2}) if
$\{h_{1},...,h_{p_1}\}\subset \{m_1'',...,m_{s-r+l}''\}$.  There is
always a (non-unique) permutation $\sigma\in \mathcal{S}_{s-r+l}$ with
$\rho= (m_{\sigma(1)}''{+}1,...,m_{\sigma(s-r+l)}''{+}1)$.  Let
$\{h_{1},...,h_{p_1}\}$ be of type $\rho_1$ for a partition
$\rho_1\in \mathsf{P}_{p_1}(2p_1-1)$. There is a (non-unique)
permutation $\tilde{\sigma}\in \mathcal{S}_{p_1}$ with
$\rho_1=(h_{\tilde{\sigma }(1)}{+}1,...,h_{\tilde{\sigma}(p_1)}{+}1)$.
Similar to the discussion before, there are $\binom{\rho}{\rho_1}$
possible choices of subpartitions $\rho_1 \subset \rho$. Every such
choice is a bijection between
$(\tilde{\sigma }(1),...,\tilde{\sigma}(p_1))$ and a subset
(consisting of $p_1$ elements) of
$(\sigma(1),...,\sigma(s-r+l))$. Resolving the permutations
$\sigma,\tilde{\sigma}$ shows that any choice of a subpartition
$\rho_1 \subset \rho$ with $\rho_1\in \mathsf{P}_{p_1}(2p_1-1)$
defines a bijection $h_1=m''_{\pi(1)}, \dots, h_{p_1}=m''_{\pi(p_1)}$
between subsets of $p_1$ elements.  But this induces a bijection
$(h_1,J_1''')= (m''_{\pi(1)},J''_{\pi(1)}),\dots,
(h_{p_1},J_{p_1}''')= (m''_{\pi(p_1)},J''_{\pi(p_1)})$, which implies
$I_1'=J''_{\pi(1)}\cup ...\cup J''_{\pi(p_1)}$.  In summary,
specifying $p_1$, any choice of a subpartition $\rho_1 \subset \rho$
with $\rho_1\in \mathsf{P}_{p_1}(2p_1-1)$ uniquely defines the factor
$\frac{1}{dx(u)}\omega^{(0)}_{|I'_{1}|+1}(u,I'_1)$ in
(\ref{omega-part-2}) whose expansion (\ref{involution-expand}) into
$p_1$ terms produces a unique subset of the terms in
(\ref{omega-part-1}). Since (\ref{involution-expand}) is a sum over
unordered partitions which we sort into the ordered
(\ref{omega-part-1}), there is a weight factor $(p_1-1)!$ for every
choice of $\rho_1\in \mathsf{P}_{p_1}(2p_1-1)$. The weight is the same
for all $\rho_1$, i.e.\
$\frac{1}{dx(u)}\omega^{(0)}_{|I'_{1}|+1}(u,I'_1)$ with sum over
compatible $I_1'$ contributes $(p_1-1)! \binom{\rho}{\rho_1}$
compatible subsets to (\ref{omega-part-1}), with summation over $p_1$
(which cannot be too large, see below) and subpartitions
$\rho_1\subset \rho$.

Continuing in the same manner, the next factor 
$\frac{1}{dx(u)}\omega^{(0)}_{|I'_{2}|+1}(u,I'_2)$ with sum over
compatible $I_2'$ contributes
$(p_2-1)! \binom{\rho\setminus \rho_1 }{\rho_2}$ compatible subsets to 
(\ref{omega-part-1}), with summation over $p_2$ and subpartitions
$\rho_2\subset \rho\setminus \rho_1$ with
$\rho_2\in \mathsf{P}_{p_2}(2p_2-1)$. And so on until 
$\frac{1}{dx(u)} \omega^{(0)}_{|I'_{l-1}|+1}(u,I'_{l-1})$.

For last factor $\frac{1}{dx(u)}\omega^{(0)}_{|I'_{l}|+1}(u,I'_{l})$
there is no choice anymore; its expansion
(\ref{involution-expand}) must produce the remaining
$p_l:=s-r+l-p_1-...-p_{l-1}$ factors
$\nabla^{h_*}\omega^{(0)}_{|J_*|+1}(\iota u,J_*)$ of 
(\ref{omega-part-1}) not obtained before. The need to have
$p_1,...,p_l \geq 1$ restricts the previous $p_i$. There is again
a factor $(p_l-1)!$ from the rearrangement in increasing order,
but no choice anymore for a subpartition. 

In conclusion, we have
\begin{align}
  (r-l)! \binom{\nu}{\mu} \cdot
  (p_1{-}1)! \binom{\nu\setminus \mu}{\rho_1}
  \cdots (p_{l-1}{-}1)! \binom{\nu\setminus (\mu{\cup} \rho_1{\cup} ...{\cup} \rho_{l-2})}{\rho_{l-1}} \cdot 
  (p_l{-}1)!
  \label{weight-partition}
\end{align}
terms (\ref{omega-part-2})
whose expansion (\ref{involution-expand}) produces (among others) the
terms in (\ref{omega-part-1}). Summing the weights 
(\ref{weight-partition}) over admissible $r$, subpartitions $\mu
\in \mathsf{P}_{r-l}(2r-2l-k)$,
size distributions  $(r-l)+p_1+...+p_l=s$ and further subpartitions
$\rho_i\in \mathsf{P}_{p_i}(2p_i-1)$ must yield the same factor $s!$ as in
(\ref{omega-part-1}). Since
\begin{align*}
\binom{\nu}{\mu,\rho_1,...,\rho_l}=
  \binom{\nu}{\mu} \cdot
 \binom{\nu\setminus \mu}{\rho_1}
 \cdots \binom{\nu\setminus (\mu{\cup} \rho_1{\cup} ...{\cup}
   \rho_{l-2})}{\rho_{l-1}}
\end{align*}
is the multinomial coefficient of partitions (where
$\rho_l:=\nu\setminus (\mu\cup \rho_1\cup ... \cup \rho_{l-1}))$ we arrive
at the formulation in Theorem~\ref{conj:main} in the Introduction.

This theorem is proved in the next section, and thus this completes the proof that $\omega^{(0)}_n$ are
symmetric in all arguments.
\end{proof}

\begin{example}
  Let $k=0$, $l=2$, $s=4$ and $\nu=(3,1,1,1)$. 
  The following size decompositions are possible:
  $((r-l)+p_1+p_2)\in \{(0+3+1),(0+2+2),(0+1+3),(1+2+1),(1+1+2),(2+1+1)\}$.

  For the first three cases $r-l=0$, $\mu$ is the empty partition
  and hence omitted. In the second case $(0+2+2)$ we would
  get $\rho_1=\rho_2=(2,1)
  \in \mathsf{P}_2(3)$, which are not subpartitions of $\nu$, so that
  there is no contribution from  $(0+2+2)$.
  In the case $(0+3+1)$ we have $\rho_2=(1)$ and $\rho_1 = (3,1,1)$ or
  $\rho_1 = (2,2,1)$, where the latter is discarded because not a subpartition
  of $\nu$. The contribution of $(0+3+1)$ is thus
  \[
\binom{3}{2,1}\binom{0}{0,0}\binom{1}{1,0}    
    (3-1)!(1-1)! =6\;.    
\]
The $k^{\text{th}}$  multinomial coefficients, here with $k\in \{1,2,3\}$,
reflects the occurrences of $k$ in the partitions.
Clearly, the same contribution is obtained for the case 
$(0+1+3)$.

In the two cases $(1+2+1)$ and $(1+1+2)$ we have one partition
$\rho_i =(2,1)\in \mathsf{P}_2(3)$ that is not a subpartition of
$\nu$, so there is no contribution.

Remains $(2+1+1)$ with $\rho_1=\rho_2=(1)\in \mathsf{P}_1(1)$ and
$\mu \in \mathsf{P}_2(4)$ with solution
$\mu=(3,1)$ or $\mu=(2,2)$. The latter must again be discarded
because it is not a subpartition of $\nu$.
The contribution of $(2+1+1)$ is thus
  \[
\binom{3}{1,1,1}\binom{0}{0,0,0} \binom{1}{1,0,0}
2! (1-1)! (1-1)! =12\;.    
\]
Summing everything, we have $6+0+6+0+0+12 =24$ as required.
\end{example}
\noindent
This example captures the cancellation between, for example,
\[
  4! \nabla^0 \omega^{(0)}_{|J_1|+1}(\iota u,J_1)
  \nabla^0 \omega^{(0)}_{|J_2|+1}(\iota u,J_2)  
 \nabla^2 \omega^{(0)}_{|J_3|+1}(\iota u,J_3)  
\nabla^0 \omega^{(0)}_{|J_4|+1}(\iota u,J_4)  \;,
\]
assuming $J_1<J_2<J_3<J_4$, and the restriction to the same
$\nabla^n$-assignment of the expansion of
\begin{align*}
  &
  \frac{\omega^{(0)}_{|J_1|+1}(u,J_1)}{dx(u)}
  \frac{\omega^{(0)}_{|I_2'|+1}(u,I_2')}{dx(u)}
  \Big|_{I_2'=J_2\cup J_3 \cup J_4}
+   \frac{\omega^{(0)}_{|I_2'|+1}(u,I_1')}{dx(u)}
\Big|_{I_1'=J_2\cup J_3 \cup J_4}
  \frac{\omega^{(0)}_{|J_1|+1}(u,J_1)}{dx(u)}
\\ 
&+\frac{\omega^{(0)}_{|I_1'|+1}(u,I_1')}{dx(u)}
\Big|_{I_1'=J_1\cup J_3 \cup J_4}
\frac{\omega^{(0)}_{|J_2|+1}(u,J_2)}{dx(u)}
+\frac{\omega^{(0)}_{|J_2|+1}(u,J_2)}{dx(u)}
\frac{\omega^{(0)}_{|I_2'|+1}(u,I_2')}{dx(u)}
\Big|_{I_2'=J_1\cup J_3 \cup J_4}
\\
&+ \frac{\omega^{(0)}_{|I_1'|+1}(u,I_1')}{dx(u)}
\Big|_{I_1'=J_1\cup J_2 \cup J_3}
\frac{\omega^{(0)}_{|J_4|+1}(u,J_4)}{dx(u)}
+\frac{\omega^{(0)}_{|J_4|+1}(u,J_4)}{dx(u)}
\frac{\omega^{(0)}_{|I_2'|+1}(u,I_2')}{dx(u)}
\Big|_{I_2'=J_1\cup J_2 \cup J_3}
\\
  &+ 2! 
  \nabla^0 \omega^{(0)}_{|J_1|+1}(\iota u,J_1)  
 \nabla^2 \omega^{(0)}_{|J_3|+1}(\iota u,J_3)  
 \Big(
 \frac{\omega^{(0)}_{|J_2|+1}(u,J_2)}{dx(u)}
 \frac{\omega^{(0)}_{|J_4|+1}(u,J_4)}{dx(u)}
 + \{J_2\leftrightarrow J_4\}\Big)
 \\
 &+ 2! 
  \nabla^0 \omega^{(0)}_{|J_2|+1}(\iota u,J_2)  
 \nabla^2 \omega^{(0)}_{|J_3|+1}(\iota u,J_3)  
\Big( \frac{\omega^{(0)}_{|J_1|+1}(u,J_1)}{dx(u)}
\frac{\omega^{(0)}_{|J_4|+1}(u,J_4)}{dx(u)}
 + \{J_1\leftrightarrow J_4\}\Big)
 \\
 &+ 2!
 \nabla^2 \omega^{(0)}_{|J_3|+1}(\iota u,J_3)  
  \nabla^0 \omega^{(0)}_{|J_4|+1}(\iota u,J_4)  
\Big( \frac{\omega^{(0)}_{|J_1|+1}(u,J_1)}{dx(u)}
\frac{\omega^{(0)}_{|J_2|+1}(u,J_2)}{dx(u)}
+ \{J_1\leftrightarrow J_2\}\Big)
\end{align*}
when inserting (\ref{involution-expand}).
We give one more example for $s=5$:

\begin{example}
  Let $k=1$, $l=3$, $s=5$ and $\nu=(2,1,1,1,1)$.
  The following size decompositions are possible:
  $((r-l)+p_1+p_2+p_2)\in \{(1+2+1+1),(1+1+2+1),(1+1+1+2),(2+1+1+1)\}$.

  For the case $(1+2+1+1)$ we have $\mu=\rho_2=\rho_3=(1)\in \mathsf{P}_1(1)$
  and $\rho_1=(2,1)\in \mathsf{P}_2(3)$. Its contribution is
  \[
   \binom{4}{1,1,1,1} \binom{1}{0,1,0,0}
    1! (2-1)!(1-1)!(1-1)! = 24 \;.
\]
Clearly, the cases $(1+1+2+1)$ and $(1+1+1+2)$ contribute the same weight.

Remains the case $(2+1+1+1)$ with $\rho_1=\rho_2=\rho_3=(1)\in \mathsf{P}_1(1)$
and $\mu=(2,1)\in  \mathsf{P}_2(3)$. Its contribution is
  \[
   \binom{4}{1,1,1,1} \binom{1}{1,0,0,0}
    2!(1-1)!(1-1)!(1-1)!=48\;.
\]
This sums to $24+24+24+48=5!$ as expected.
\end{example}

\section{Combinatorial argument}

\label{sec:combinatorial-proof}

The goal of this section is to give a proof of Theorem~\ref{conj:main}.
First of all, we discuss the structure of the partitions in $\mathsf{P}_s(2s-k-l)$ for $0\leq k+l\leq s$. Conveniently, any partition can be represented as a formal product $\prod_{i=1}^\infty \underline{i}\,^{n_i}$, meaning that it has $n_i$ parts of length $i$, $i=1,2,3,\dots$, and we use this notation throughout this section (for instance, the product of two partitions in this notation is the union of the partitions in the standard notation).

\begin{lemma} If $\nu\in \mathsf{P}_s(2s-k-l)$, then it can be written as 
	\[
	\nu = \underline{1}\,^{k+l} \prod_{j=1}^{N} (\underline{1}\,^{a_j-2}\,\underline{a_j})
	\]
	for some $a_1,\dots,a_{N}\geq 2$ such that $k+l+\sum_{j=1}^N (a_j-1) = s$.
\end{lemma}

\begin{proof} First, notice that if $\nu\in \mathsf{P}_s(2s)$ is equal to $(a_1,\dots,a_N,1,\dots,1)$, where $a_i\geq 2$ for $i=1,\dots,N$, then the number of $1$'s must be equal to $\sum_{i=1}^N (a_i-2)$. 
	
	Then we can proceed by induction on $k+l$, with the induction hypothesis being the statement of the lemma. To this end, just notice that if $\nu\in \mathsf{P}_s(2s-k-l)$ and $s\geq k+l\geq 1$, then at least one part of $\nu$ is equal to $1$. Then we can apply the induction hypothesis to $\nu\cdot\underline{1}^{-1}\in \mathsf{P}_{s-1}(2(s-1)-k-l+1)$ (here we let $\nu\cdot\underline{1}^{-1}$ denote the partition $\nu$ with one part of length $1$ removed).
\end{proof}

With this observation there is the following equivalent reformulation of Equation~\eqref{eq:Conjecture-Partitions}.

\begin{corollary}\label{cor:Equivalent-Reformulation} Equation~\eqref{eq:Conjecture-Partitions} is equivalent to the following identity:
	\begin{align} 
		s! & = \sum_{\footnotesize \substack {I_0\sqcup I_1\sqcup \cdots \sqcup I_l\\ = \{1,\dots,N\}}}
		\binom{k+l+\textstyle\sum\limits_{j=1}^N (a_j-2)}{\ k+\textstyle\sum\limits_{j\in I_0} (a_j-2),\ 1+\textstyle\sum\limits_{j\in I_1} (a_j-2),\ \dots\ ,\ 1+\textstyle\sum\limits_{j\in I_l} (a_j-2)\ }
		\notag \\ & \qquad \times
		\big(k+\textstyle\sum\limits_{j\in I_0} (a_j-1)\big)! \prod_{i=1}^l \big(\textstyle\sum\limits_{j\in I_i} (a_j-1)\big)!,
		\label{eq:CombinatorialReformulation}
	\end{align}
where $I_0,...,I_N$ are allowed to be empty.
\end{corollary}

\begin{proof} Note that the coefficient $ \binom{\nu}{\mu, \rho_1,...,\rho_l} $ is non-trivial if an only if $\nu = \mu\cdot \prod_{i=1}^l \rho_i$. Moreover, if $\nu = \underline{1}\,^{k+l} \prod_{j=1}^{N} (\underline{1}\,^{a_j-2}\,\underline{a_j})$, then $\mu = \underline{1}\,^{k} \prod_{j\in I_0} (\underline{1}\,^{a_j-2}\,\underline{a_j})$ and $\rho_i= \underline{1}\, \prod_{j\in I_i} (\underline{1}\,^{a_j-2}\,\underline{a_j})$, $i=1,\dots,l$, for some $I_0\sqcup \cdots \sqcup I_l = \{1,\dots,N\}$. In particular, in this case
	\begin{align*}
		k+\textstyle\sum\limits_{j\in I_0} (a_j-1) & = r-l; \\
		\textstyle\sum\limits_{j\in I_i} (a_j-1) & = p_i-1, & i&=1,\dots,l.
	\end{align*} 
	Note also that the number of splits $I_0\sqcup \cdots \sqcup I_l = \{1,\dots,N\}$ that gives exactly the same tuples of partitions $(\mu,\rho_1,\dots,\rho_l)$ is equal to $\binom{\tilde \nu }{\tilde \mu, \tilde \rho_1,...,\tilde \rho_l} $, where $\tilde \nu$ (resp., $\tilde \mu$, $\tilde \rho_i$) is equal to the partition $\nu$ (resp., $\tilde \mu$, $\tilde \rho_i$) with all parts of length $1$ removed. The observation that 
	\begin{align*}
		&
		\dfrac {\displaystyle\binom{ \nu }{\mu,\rho_1,..., \rho_l} } 
		{\displaystyle \binom{\tilde \nu }{\tilde \mu, \tilde \rho_1,...,\tilde \rho_l} } = 
		\binom{k+l+\textstyle\sum\limits_{j=1}^N (a_j-2)}{\ k+\textstyle\sum\limits_{j\in I_0} (a_j-2),\ 1+\textstyle\sum\limits_{j\in I_1} (a_j-2),\ \dots\ ,\ 1+\textstyle\sum\limits_{j\in I_l} (a_j-2)\ }
	\end{align*}
	completes the proof. 
\end{proof}

Equation \eqref{eq:CombinatorialReformulation} can be derived by induction from its special case for $l=1$: 

\begin{lemma}\label{lem:Comb-l=1} The following identity holds:
	\begin{align} 
		\big(1+k+\textstyle\sum\limits_{j=1}^N (a_j-1)\big)! & = \sum_{\footnotesize \substack {I_0\sqcup I_1 = \\ \{1,\dots,N\}}}
		\binom{k+1+\textstyle\sum\limits_{j=1}^N (a_j-2)}{\ k+\textstyle\sum\limits_{j\in I_0} (a_j-2),\ 1+\textstyle\sum\limits_{j\in I_1} (a_j-2)\ }
		\notag \\ & \qquad \times
		\big(k+\textstyle\sum\limits_{j\in I_0} (a_j-1)\big)! \big(\textstyle\sum\limits_{j\in I_1} (a_j-1)\big)! 
		\label{eq:CombinatorialReformulation-l=1}
	\end{align}	
\end{lemma}

\begin{proof}
	Denote $M\coloneqq N+1$, $b_j\coloneqq a_j-2$ for $j=1,\dots,M-1$, and $b_{M}\coloneqq k-1$. Then we can rewrite~\eqref{eq:CombinatorialReformulation-l=1} in a much more symmetric way, namely,
	\begin{align*}
		\big(1+M+{\textstyle\sum\limits_{j=1}^M} b_j \big)! & = \frac 12 \sum_{\footnotesize \substack {I\sqcup J = \\ \{1,\dots,M\}}} \frac{\big(2+\textstyle\sum\limits_{j=1}^M b_j\big)!}{\big(1+\textstyle\sum\limits_{j\in I} b_j \big)!\big(1+\textstyle\sum\limits_{j\in J} b_j \big)!}
		\big(|I|+\textstyle\sum\limits_{j\in I} b_j \big)!\big(|J|+\textstyle\sum\limits_{j\in J} b_j \big)!,
	\end{align*}
	and furthermore as an identity for the polynomials in $b_1,\dots,b_M$ of degree $M-1$:
	\begin{align} \label{eq:Reformulation-l=1}
		\frac{\big(1+M+{\textstyle\sum\limits_{j=1}^M} b_j \big)!}{\big(2+\textstyle\sum\limits_{j=1}^M b_j\big)!} = \frac 12 \sum_{\footnotesize \substack {I\sqcup J = \\ \{1,\dots,M\}}} \frac{\big(|I|+\textstyle\sum\limits_{j\in I} b_j \big)!}{\big(1+\textstyle\sum\limits_{j\in I} b_j \big)!}
		\cdot 
		\frac{\big(|J|+\textstyle\sum\limits_{j\in J} b_j \big)!}{\big(1+\textstyle\sum\limits_{j\in J} b_j \big)!}.
	\end{align}
	In this form, it is obvious for $M=0,1$ and for any $M\geq 2$, $b_1,\dots,b_M=0$. In order to establish this identity in general, we proceed by induction on $M$. We denote the left hand side by $L_M(b_1,\dots,b_M)$ and the right hand side by $R_M(b_1,\dots,b_M)$ and show that for any $i=1,\dots,M$
	\begin{align} \label{eq:L-L=R-R}
		& L_M(b_1,\dots,b_M)-L_M(b_1,\dots,b_{i-1},b_i-1,b_{i+1},\dots,b_M) 
		\\ \notag
		& = R_M(b_1,\dots,b_M)-R_M(b_1,\dots,b_{i-1},b_i-1,b_{i+1},\dots,b_M)
	\end{align}
	under the assumption that the identity holds for any parameters for $M-1$. Note that since both $L_M$ and $R_M$ are symmetric in their arguments, it is sufficient to prove~\eqref{eq:L-L=R-R} for $i=M$. To this end, we have:
	\begin{align*}
		& L_M(b_1,\dots,b_M)-L_M(b_1,\dots,b_M-1) 
		\\ & = 	\frac{\big(1+M+{\textstyle\sum\limits_{j=1}^M} b_j \big)!}{\big(2+\textstyle\sum\limits_{j=1}^M b_j\big)!}
		- \frac{\big(M+{\textstyle\sum\limits_{j=1}^M} b_j \big)!}{\big(1+\textstyle\sum\limits_{j=1}^M b_j\big)!} 
		=  (M-1) \frac{\big(M+{\textstyle\sum\limits_{j=1}^M} b_j \big)!}{\big(2+\textstyle\sum\limits_{j=1}^M b_j\big)!}
		\\ &
		= \sum_{i=1}^{M-1} L_{M-1}(b_1,\dots,b_{i-1},b_i+b_M,b_{i+1},\dots,b_{M-1}).
	\end{align*}
	On the other hand, with a similar computation,
	\begin{align*}
		& R_M(b_1,\dots,b_M)-R_M(b_1,\dots,b_M-1) 	\\
		& = \frac 12 \sum_{\footnotesize \substack {I\sqcup J = \\ \{1,\dots,M\} \\ I\ni M}} ({|I|-1})\cdot \frac{\big(|I|-1+\textstyle\sum\limits_{j\in I} b_j \big)!}{\big(1+\textstyle\sum\limits_{j\in I} b_j \big)!}
		\cdot 
		\frac{\big(|J|+\textstyle\sum\limits_{j\in J} b_j \big)!}{\big(1+\textstyle\sum\limits_{j\in J} b_j \big)!}
		\\ & \qquad 
		+ \frac 12 \sum_{\footnotesize \substack {I\sqcup J = \\ \{1,\dots,M\} \\ J\ni M}} (|J|-1)\cdot \frac{\big(|I|+\textstyle\sum\limits_{j\in I} b_j \big)!}{\big(1+\textstyle\sum\limits_{j\in I} b_j \big)!}
		\cdot 
		\frac{\big(|J|-1+\textstyle\sum\limits_{j\in J} b_j \big)!}{\big(1+\textstyle\sum\limits_{j\in J} b_j \big)!}
		\\ & = \sum_{i=1}^{M-1} R_{M-1}(b_1,\dots,b_{i-1},b_i+b_M,b_{i+1},\dots,b_{M-1}).
	\end{align*}
	By induction assumption it implies that~\eqref{eq:L-L=R-R} holds, and since
	\[
	L_M(0,\dots,0)= R_M(0,\dots,0)
	\]
	we conclude that
	\[
	L_M(b_1,\dots,b_M)= R_M(b_1,\dots,b_M)
	\]
	for any values of $b_1,\dots,b_M$. (Formally speaking, the difference equation implies this equality only for any integer arguments $b_i$'s, but since it is an equality of polynomials and integers form a Zariski dense set, we have this equality for any complex values of $b_i$'s). 
	
	Thus~\eqref{eq:Reformulation-l=1} holds, and therefore~\eqref{eq:CombinatorialReformulation-l=1} holds as well. 
\end{proof}

\begin{corollary} \label{cor:ReformulationHolds} Equation~\eqref{eq:CombinatorialReformulation} holds for any $l\geq 0$.
\end{corollary}

\begin{proof} We proceed by induction on $l$. For $l=0$ Equation~\eqref{eq:CombinatorialReformulation} reads $s!=s!$ and for $l=1$ it is proved in Lemma~\ref{lem:Comb-l=1}. Assume it is proved for some $l$. Then, for $l+1$ we have:
	\begin{align} 
		& \sum_{\footnotesize \substack {I_0\sqcup I_1\sqcup \cdots \sqcup I_{l+1}\\ = \{1,\dots,N\}}}
		\binom{k+l+1+\textstyle\sum\limits_{j=1}^N (a_j-2)}{\ k+\textstyle\sum\limits_{j\in I_0} (a_j-2),\ 1+\textstyle\sum\limits_{j\in I_1} (a_j-2),\ \dots\ ,\ 1+\textstyle\sum\limits_{j\in I_{l+1}} (a_j-2)\ }
		\notag \\ \notag & \qquad \times
		\big(k+\textstyle\sum\limits_{j\in I_0} (a_j-1)\big)! \prod_{i=1}^{l+1} \big(\textstyle\sum\limits_{j\in I_i} (a_j-1)\big)! 
		\\ \notag & 
		= \sum_{\footnotesize \substack {J \sqcup I_{l+1}\\ = \{1,\dots,N\}}}
		\binom{k+l+1+\textstyle\sum\limits_{j=1}^N (a_j-2)}{\ k+l+\textstyle\sum\limits_{j\in J} (a_j-2) ,\ 1+\textstyle\sum\limits_{j\in I_{l+1}} (a_j-2)\ }
		\\ \notag & \qquad \bigg[
		\sum_{\footnotesize \substack {I_0\sqcup I_1\sqcup \cdots \sqcup I_{l}\\ = J}} \binom{k+l+\textstyle\sum\limits_{j\in J} (a_j-2)}{\ k+\textstyle\sum\limits_{j\in I_0} (a_j-2),\ 1+\textstyle\sum\limits_{j\in I_1} (a_j-2),\ \dots\ ,\ 1+\textstyle\sum\limits_{j\in I_{l}} (a_j-2)\ }
		\notag \\ \notag & \qquad \times
		\big(k+\textstyle\sum\limits_{j\in I_0} (a_j-1)\big)! \prod_{i=1}^{l} \big(\textstyle\sum\limits_{j\in I_i} (a_j-1)\big)! \bigg] \cdot \big(\textstyle\sum\limits_{j\in I_{l+1}} (a_j-1)\big)!.
	\end{align}
	Note that by the induction assumption, the expression in the square brackets is equal to $\big(k+l+\textstyle\sum_{j\in J} (a_j-1)\big)!$. Thus, the full expression above is equal to 
	\begin{align*} 
		&  \sum_{\footnotesize \substack {J \sqcup I_{l+1}\\ = \{1,\dots,N\}}}
		\binom{k+l+1+\textstyle\sum\limits_{j=1}^N (a_j-2)}{\ k+l+\textstyle\sum\limits_{j\in J} (a_j-2) ,\ 1+\textstyle\sum\limits_{j\in I_{l+1}} (a_j-2)\ }
		\\ \notag & \times
		\big(k+l+\textstyle\sum\limits_{j\in J} (a_j-1)\big)!  \big(\textstyle\sum\limits_{j\in I_{l+1}} (a_j-1)\big)!,
	\end{align*}
	which is equal to $\big(k+l+1+\sum_{j=1}^N (a_j-1)\big)!$ by Lemma~\ref{lem:Comb-l=1}.
\end{proof}

Now we are ready to complete the proof of Theorem~\ref{conj:main}.

\begin{proof}[Proof of Theorem~\ref{conj:main}]
	By Corollary~\ref{cor:Equivalent-Reformulation}, the statement of the theorem is equivalent to Equation~\eqref{eq:CombinatorialReformulation}. The latter equation holds by Corollary~\eqref{cor:ReformulationHolds}
\end{proof}

\section{Outlook}

In \cite{Hock:2021tbl} the first and the last named authors have introduced and studied a family of 
meromorphic differentials $\omega^{(0)}_n(z_1,...,z_n)$ defined via
(\ref{om02}) and an involution identity (\ref{eq:flip-om}), where
$\iota z=\frac{az+b}{cz-a}$ for $a^2+bc\neq 0$ is a holomorphic involution
of $\mathbb{P}^1$. In this paper we settle the so-far open question about the symmetry
of the $\omega^{(0)}_n(z_1,...,z_n)$ by its reduction to  a combinatorial
Theorem~\ref{conj:main} about integer partitions. 


As also shown in \cite{Hock:2021tbl}, the quartic analogue of the Kontsevich
matrix model is an example where the
$\omega^{(0)}_n(z_1,...,z_n)$ are of this form. In \cite{Hock:2023nki}  the first and the last named authors succeeded
in proving recursion formulae also for $\omega^{(1)}_n(z_1,...,z_n)$ of genus $g=1$
by another strategy (extended loop equations). The proof that also
the  $\omega^{(1)}_n(z_1,...,z_n)$ are symmetric is left for the future.

\section*{Acknowledgements}

AH was supported through the Walter-Benjamin fellowship\footnote{\
  ``Funded by
the Deutsche Forschungsgemeinschaft (DFG, German Research
Foundation) -- Project-ID 465029630 and Project-ID 51922044''} and
through the project \ ``Topological Recursion, Duality and
Applications''\footnote{\ ``Funded by
the Deutsche Forschungsgemeinschaft (DFG, German Research
Foundation) -- Project-ID 551478549''}. 
SS was supported by the Dutch Research Council. 
RW  was
supported\footnote{\ ``Funded by
  the Deutsche Forschungsgemeinschaft (DFG, German Research
  Foundation) -- Project-ID 427320536 -- SFB 1442, as well as under
  Germany's Excellence Strategy EXC 2044 390685587, Mathematics
  M\"unster: Dynamics -- Geometry -- Structure''} by the Cluster of
Excellence \emph{Mathematics M\"unster} and the CRC 1442 \emph{Geometry:
  Deformations and Rigidity}.


\begin{thebibliography}{BHW22}
\expandafter\ifx\csname url\endcsname\relax
  \def\url#1{\texttt{#1}}\fi
\expandafter\ifx\csname doi\endcsname\relax
  \def\doi#1{\burlalt{doi:#1}{http://dx.doi.org/#1}}\fi
\expandafter\ifx\csname urlprefix\endcsname\relax\def\urlprefix{URL }\fi
\expandafter\ifx\csname href\endcsname\relax
  \def\href#1#2{#2}\fi
\expandafter\ifx\csname burlalt\endcsname\relax
  \def\burlalt#1#2{\href{#2}{#1}}\fi

\bibitem[BHW22]{Branahl:2020yru}
J.~Branahl, A.~Hock, and R.~Wulkenhaar.
\newblock {Blobbed topological recursion of the quartic Kontsevich model I:
  Loop equations and conjectures}.
\newblock {\em Commun. Math. Phys.}, 393(3):1529--1582, 2022,
  \burlalt{2008.12201}{http://arxiv.org/abs/2008.12201}.
\newblock \doi{10.1007/s00220-022-04392-z}.

\bibitem[BS17]{Borot:2015hna}
G.~Borot and S.~Shadrin.
\newblock {Blobbed topological recursion: properties and applications}.
\newblock {\em Math. Proc. Cambridge Phil. Soc.}, 162(1):39--87, 2017,
  \burlalt{1502.00981}{http://arxiv.org/abs/1502.00981}.
\newblock \doi{10.1017/S0305004116000323}.

\bibitem[EO07]{Eynard:2007kz}
B.~Eynard and N.~Orantin.
\newblock {Invariants of algebraic curves and topological expansion}.
\newblock {\em Commun. Num. Theor. Phys.}, 1:347--452, 2007,
  \burlalt{math-ph/0702045}{http://arxiv.org/abs/math-ph/0702045}.
\newblock \doi{10.4310/CNTP.2007.v1.n2.a4}.

\bibitem[HW23]{Hock:2023nki}
A.~Hock and R.~Wulkenhaar.
\newblock {Blobbed topological recursion from extended loop equations}.
\newblock 2023, \burlalt{2301.04068}{http://arxiv.org/abs/2301.04068}.

\bibitem[HW25]{Hock:2021tbl}
A.~Hock and R.~Wulkenhaar.
\newblock {Blobbed topological recursion of the quartic Kontsevich model II:
  Genus=0}.
\newblock {\em Ann. Inst. H. Poincare D Comb. Phys. Interact.}, online first,
  2025, \burlalt{2103.13271}{http://arxiv.org/abs/2103.13271}.
\newblock \doi{10.4171/AIHPD/198}.
\newblock with an appendix by M.\ Do\l{}\k{e}ga.

\end{thebibliography}

\end{document}